\documentclass{amsart}

\usepackage[all]{xy}

\usepackage{xr-hyper}

\usepackage{hyperref}

\theoremstyle{plain}
\newtheorem{theorem}[subsection]{Theorem}

\newtheorem{lemma}[subsection]{Lemma}
\newtheorem{corollary}[subsection]{Corollary}

\theoremstyle{definition}
\newtheorem{definition}[subsection]{Definition}

\newtheorem{situation}[subsection]{Situation}

\theoremstyle{remark}
\newtheorem{remark}[subsection]{Remark}

\numberwithin{equation}{section}

\begin{document}

\title{Invariant rings through categories}

\author[Alper]{Jarod Alper}
\author[de Jong]{A. J. de Jong}

\address{Department of Mathematics\\
Columbia University\\
2990 Broadway\\
New York, NY 10027}
\email{jarod@math.columbia.edu}

\address{Department of Mathematics\\
Columbia University\\
2990 Broadway\\
New York, NY 10027}
\email{dejong@math.columbia.edu}

\maketitle

\begin{abstract}
We formulate a notion of ``geometric reductivity" in an abstract categorical setting which we refer to as adequacy.  The main theorem states that the adequacy condition implies that the ring of invariants is finitely generated.  This result applies to the category of modules over a bialgebra,
the category of comodules over a bialgebra, and the category
of quasi-coherent sheaves on a finite type algebraic stack over an
affine base. 
\end{abstract}

\setcounter{tocdepth}{1}
\tableofcontents

\section{Introduction}
\label{section-introduction}

\noindent
A fundamental theorem in invariant theory states that if a reductive group
$G$ over a field $k$ acts on a finitely generated $k$-algebra $A$, then the
ring of invariants $A^G$ is finitely generated over $k$
(see \cite[Appendix 1.C]{git3}).  Mumford's conjecture, proven by Haboush
in \cite{haboush}, states that reductive groups are geometrically
reductive; therefore this theorem is reduced to showing that the ring
of invariants under an action by a geometrically reductive group is
finitely generated, which was originally proved by Nagata in
\cite{nagata_invariants-affine}.  

\medskip\noindent
Nagata's theorem has been generalized to various settings.  Seshadri showed
an analogous result for an action of a ``geometrically reductive" group
scheme over a universally Japanese base scheme
(see \cite{seshadri_reductivity}).  In \cite{ferrer-santos}, the result
is generalized to an action of a ``geometrically reductive" commutative
Hopf algebra over a field on a coalgebra.  In \cite{kalniuk-tyc}, an
analogous result is proven for an action of  a ``geometrically reductive"
(non-commutative) Hopf algebra over a field on an algebra.  In
\cite{alper_good_arxiv} and \cite{alper_adequate}, analogous results are
 shown for the invariants of certain pre-equivalence relations; 
 moreover, \cite{alper_adequate} systematically develops the theory of
 adequacy for algebraic stacks.

\medskip\noindent
These settings share a central underlying ``adequacy'' property which we
formulate in an abstract categorical setting. Namely, consider a homomorphism
of commutative rings $R \to A$. Consider an $R$-linear $\otimes$-category
$\mathcal{C}$ with a faithful exact $R$-linear $\otimes$-functor
$$
F : \mathcal{C} \longrightarrow \text{Mod}_A
$$
such that $\mathcal{C}$ is endowed with a ring object
$\mathcal{O} \in \text{Ob}(\mathcal{C})$ which is a unit for
$\otimes$. For precise definitions, please see
Situation \ref{situation-category}.
One can then define
$$
\Gamma : \mathcal{C} \longrightarrow \text{Mod}_R,
\quad
\mathcal{F}
\longmapsto
\text{Mor}_{\mathcal{C}}(\mathcal{O}, \mathcal{F}).
$$
%$\Gamma(\mathcal{F}) = \text{Mor}_{\mathcal{C}}(\mathcal{O}, \mathcal{F})$.
Adequacy means (roughly) in this setting that $\Gamma$ satisfies: if
$\mathcal{A} \to \mathcal{B}$ is a surjection of commutative ring objects
and if $f \in \Gamma(\mathcal{B})$, then there exists
$g \in \Gamma(\mathcal{A})$ with $g \mapsto f^n$ for some $n > 0$.
The main theorem of this paper is Theorem \ref{theorem-finite-type} 
which states (roughly)
that if $\Gamma$ is adequate, then
\begin{enumerate}
\item $\Gamma(\mathcal{A})$ is of finite type over $R$
if $\mathcal{A}$ is of finite type, and
\item $\Gamma(\mathcal{F})$ is a finite type $\Gamma(\mathcal{A})$-module
if $\mathcal{F}$ is of finite type.
\end{enumerate}
Note that additional assumptions have to be imposed on
the categorical setting in order to even formulate the result.

\medskip\noindent
In the final sections of this paper, we show how the abstract categorical
setting applies to (a) the category of modules over a bialgebra,
(b) the category of comodules over a bialgebra, and (c) the category
of quasi-coherent sheaves on a finite type algebraic stack over an
affine base. Thus the main theorem above unifies 
and generalizes the results mentioned
above, which was the original motivation for this research.   

\medskip\noindent
What is lacking in this theory is a practical criterion for adequacy.
Thus we would like to ask the following questions: Is there is notion
of reductivity in the categorical setting? Is there an abstract
analogue of Haboush's theorem? We hope to return to these question in
future research.

\medskip\noindent
\subsection*{Conventions} Rings are associative with $1$.
Abelian categories are additive categories with kernels and cokernels
such that $\text{Im} \cong \text{Coim}$ for any morphism.

\section{Setup}
\label{section-setup}

\noindent
In this section, we introduce the types of structure we are going to work
with. We keep the list of basic properties to an absolute minimum, and later
we introduce additional axioms to impose.

\begin{situation}
\label{situation-category}
We consider the following systems of data:
\begin{enumerate}
\item $R \to A$ is a map of commutative rings,
\item $\mathcal{C}$ is an $R$-linear abelian category,
\item $\otimes : \mathcal{C} \times \mathcal{C} \to \mathcal{C}$
is an $R$-bilinear functor,
\item $F : \mathcal{C} \to \text{Mod}_A$ is a faithful exact $R$-linear
functor,
\item there is a given bifunctorial isomorphism
$$
\gamma_{\mathcal{F}, \mathcal{G}} :
F(\mathcal{F}) \otimes_A F(\mathcal{G})
\longrightarrow
F(\mathcal{F} \otimes \mathcal{G}),
$$
\item there exist functorial isomorphisms
$$
\tau_{\mathcal{F}, \mathcal{G}, \mathcal{H}} :
(\mathcal{F} \otimes \mathcal{G}) \otimes \mathcal{H}
\longrightarrow
\mathcal{F} \otimes (\mathcal{G} \otimes \mathcal{H})
$$
which are compatible with the usual associativity
of tensor products of $A$-modules via $\gamma$, and
\item there is an object $\mathcal{O}$ of $\mathcal{C}$
endowed with functorial isomorphisms
$\mu : \mathcal{O} \otimes \mathcal{F} \to \mathcal{F}$, and
$\mu : \mathcal{F} \otimes \mathcal{O} \to \mathcal{F}$
such that $F(\mathcal{O}) = A$ and the isomorphisms
correspond to the usual isomorphisms $A \otimes_A M = M \otimes_A A = M$
(via $\gamma$ above).
\end{enumerate}
\end{situation}

\noindent
If an associativity constraint $\tau$ as above exists,
then it is uniquely determined by the condition that it agrees with
the usual associativity constraint for $A$-modules (as $F$ is faithful).
Hence we often do not list it as part of the data, and we say
``Let $(R \to A, \mathcal{C}, \otimes, F, \gamma, \mathcal{O}, \mu)$
be as in Situation \ref{situation-category}''.

\medskip\noindent
Note that in particular $\mathcal{O} \otimes \mathcal{O} = \mathcal{O}$,
and hence that $\mathcal{O}$ is a ring object of $\mathcal{C}$ (see
Section \ref{section-ring}), and for
this ring structure every object of $\mathcal{C}$ is in a canonical way
an $\mathcal{O}$-module.

\begin{definition}
\label{definition-global-sections}
In the situation above we define the {\it global sections functor}
to be the functor
$$
\Gamma : \mathcal{C} \longrightarrow \text{Mod}_R,
\quad
\mathcal{F}
\longmapsto
\Gamma(\mathcal{F}) = \text{Mor}_{\mathcal{C}}(\mathcal{O}, \mathcal{F}).
$$
\end{definition}

\noindent
Note that $\Gamma(\mathcal{F}) \subset F(\mathcal{F})$
since the functor $F$ is faithful. There are canonical maps
$
\Gamma(\mathcal{F}) \otimes_R \Gamma(\mathcal{G})
\longrightarrow
\Gamma(\mathcal{F} \otimes \mathcal{G})
$
defined by mapping the pure tensor $f \otimes g$ to the map
$$
\mathcal{O} = \mathcal{O} \otimes \mathcal{O}
\xrightarrow{f \otimes g}
\mathcal{F} \otimes_R \mathcal{G}
$$
For any pair of objects $\mathcal{F}, \mathcal{G}$ of $\mathcal{C}$
there is a commutative diagram
$$
\xymatrix{
\Gamma(\mathcal{F}) \otimes_R \Gamma(\mathcal{G}) \ar[d] \ar[r] &
\Gamma(\mathcal{F} \otimes \mathcal{G}) \ar[d] \\
F(\mathcal{F}) \otimes_A F(\mathcal{G}) \ar@{=}[r] &
F(\mathcal{F} \otimes \mathcal{G})
}
$$
In particular, there is a natural $\Gamma(\mathcal{O})$-module
structure on $\Gamma(\mathcal{F})$ for every object $\mathcal{F}$
of $\mathcal{C}$.

\section{Axioms}
\label{section-axioms}

\noindent
The following axioms will be introduced throughout the text.  
For the convenience of the reader, we list them here.

\begin{definition}
\label{definition-axioms}
Let
$(R \to A, \mathcal{C}, \otimes, F, \gamma, \mathcal{O}, \mu)$
be as in Situation \ref{situation-category}. We introduce the following
axioms:
\begin{enumerate}
\item[(D)] The category $\mathcal{C}$ has arbitrary direct
summands, and $\otimes$, $F$, and $\Gamma$ commute with these.
\item[(C)] There exist functorial isomorphisms
$\sigma_{\mathcal{F}, \mathcal{G}} : \mathcal{F} \otimes \mathcal{G}
\longrightarrow \mathcal{G} \otimes \mathcal{F}$
such that $\sigma_{\mathcal{F}, \mathcal{G}}$ is via $F$ and $\gamma$
compatible with the usual commutativity constraint
$M \otimes_A N \cong N \otimes_A M$ on $A$-modules.
\item[(I)] The category $\mathcal{C}$ has arbitrary direct
products, and $F$ commutes with them.
\item[(S)] For every object $\mathcal{F}$ of $\mathcal{C}$ and any
$n \geq 1$ there exists a quotient
$$
\mathcal{F}^{\otimes n} \longrightarrow \text{Sym}^n_{\mathcal{C}}(\mathcal{F})
$$
such that the map of $A$-modules
$
F(\mathcal{F}^{\otimes n})
\longrightarrow 
F(\text{Sym}^n_{\mathcal{C}}(\mathcal{F}))
$
factors through the natural surjection
$F(\mathcal{F})^{\otimes n} \to \text{Sym}^n_A(F(\mathcal{F}))$, and
such that $\text{Sym}^n_{\mathcal{C}}(\mathcal{F})$ is universal with this
property.
\item[(L)] Every object $\mathcal{F}$ of $\mathcal{C}$ is a filtered
colimit $\mathcal{F} = \text{colim}\ \mathcal{F}_i$
of finite type objects $\mathcal{F}_i$ such that
$F(\mathcal{F}) = \text{colim}\ F(\mathcal{F}_i)$.
\item[(N)] The ring $A$ is Noetherian.
\item[(G)] The functor $\Gamma$ is exact.
\item[(A)] For every surjection of weakly commutative ring objects 
$\mathcal{A} \to \mathcal{B}$ in $\mathcal{C}$ with $\mathcal{A}$ locally
finite, and any
$f \in \Gamma(\mathcal{B})$, there exists an $n > 0$ and an element
$g \in \Gamma(\mathcal{A})$ such that $g \mapsto f^n$ in $\Gamma(\mathcal{B})$. \end{enumerate}
\end{definition}

\noindent
Terminology used above:
An object $\mathcal{F}$ of $\mathcal{C}$ is \emph{finite type}
if $F(\mathcal{F})$ is finite type, see Definition \ref{definition-properties}.
A ring object $\mathcal{A}$, see Definition \ref{definition-ring},
is \emph{weakly commutative} if $F(\mathcal{A})$ is commutative,
see Definition \ref{definition-weakly-commutative-ring}.
An object $\mathcal{F}$ of $\mathcal{C}$ is \emph{locally finite}
if it is a filtered colimit $\mathcal{F} = \text{colim}\ \mathcal{F}_i$
of finite type objects $\mathcal{F}_i$ such that also
$F(\mathcal{F}) = \text{colim}\ F(\mathcal{F}_i)$,
see Definition \ref{definition-locally-finite}.

\section{Direct summands}
\label{section-direct-summands}

\noindent
We cannot prove much without the following axiom.

\begin{definition}
\label{definition-direct-summands}
Let $(R \to A, \mathcal{C}, \otimes, F, \gamma, \mathcal{O}, \mu)$
be as in Situation \ref{situation-category}. We introduce the following
axiom:
\begin{enumerate}
\item[(D)] The category $\mathcal{C}$ has arbitrary direct
summands, and $\otimes$, $F$, and $\Gamma$ commute with these.
\end{enumerate}
\end{definition}

\noindent
This implies that $\mathcal{C}$ has colimits and that
$\otimes$, $F$ and $\Gamma$ commute with these.

\begin{lemma}
\label{lemma-adjoint}
Assume that we are in Situation \ref{situation-category} and that the axiom
(D) holds. Then $\Gamma$ has a left adjoint 
$$
\mathcal{O} \otimes_R - : 
\text{Mod}_R 
\longrightarrow
\mathcal{C}
$$ 
with $\mathcal{O} \otimes_R R \cong \mathcal{O}$, and
$F(\mathcal{O} \otimes_R M) = A \otimes_R M$. Moreover, for any object
$\mathcal{F}$ of $\mathcal{C}$ there is a canonical isomorphism
$\mathcal{F} \otimes (\mathcal{O} \otimes_R M) =
(\mathcal{O} \otimes_R M) \otimes \mathcal{F}$ which reduces to
the obvious isomorphism on applying $F$.
\end{lemma}

\begin{proof}
For any $R$-module $M$ choose a presentation
$\bigoplus_{j \in J} R \to \bigoplus_{i \in I} R \to M \to 0$
and define
$$
\mathcal{O} \otimes_R M =
\text{Coker}(\bigoplus\nolimits_{j \in J} \mathcal{O} 
\longrightarrow
\bigoplus\nolimits_{i \in I} \mathcal{O})
$$
where the arrow is given by the same matrix as the matrix used in the
presentation for $M$. With this definition it is clear that
$F(\mathcal{O} \otimes_R M) = A \otimes_R M$. Moreover, since there
is an exact sequence
$$
\bigoplus\nolimits_{j \in J} \mathcal{O} 
\longrightarrow
\bigoplus\nolimits_{i \in I} \mathcal{O}
\longrightarrow
\mathcal{O} \otimes_R M \longrightarrow 0
$$
it is straightforward to verify that
$\text{Mor}_\mathcal{C}(\mathcal{O} \otimes_R M, \mathcal{F})
= \text{Mor}_R(M, \Gamma(\mathcal{F}))$. We leave the
proof of the last statement to the reader.
\end{proof}

\noindent
In the situation of the lemma we will write $M \otimes_R \mathcal{F}$
instead of the more clumsy notation
$M \otimes_R \mathcal{O} \otimes \mathcal{F}$.

\begin{remark}
\label{remark-commutes}
Let $(R \to A, \mathcal{C}, \otimes, F, \gamma, \mathcal{O}, \mu)$
be as in Situation \ref{situation-category}, and further
assume (D) holds. By Lemma \ref{lemma-adjoint} above,
we have a diagram of functors
$$
\xymatrix{
\text{Mod}_{R} 
\ar@/^/[rr]^-{\mathcal{O} \otimes_{R} -}
\ar[rrd]_{A \otimes_R -}
& &
\mathcal{C}
\ar[d]^-F
\ar@/^/[ll]^-\Gamma \\
& & \text{Mod}_A
}
$$
where $F \circ (\mathcal{O} \otimes_{R} -) = (A \otimes_{R} -)$, and
$\mathcal{O} \otimes_{R} -$ is a left adjoint to $\Gamma$.
\end{remark}

\section{Commutativity}
\label{section-commutative}

\begin{definition}
\label{definition-commutative}
Let $(R \to A, \mathcal{C}, \otimes, F, \gamma, \mathcal{O}, \mu)$
be as in Situation \ref{situation-category}. We introduce the following
axiom:
\begin{enumerate}
\item[(C)] There exist functorial isomorphisms
$\sigma_{\mathcal{F}, \mathcal{G}} : \mathcal{F} \otimes \mathcal{G}
\longrightarrow \mathcal{G} \otimes \mathcal{F}$
such that $\sigma_{\mathcal{F}, \mathcal{G}}$ is via $F$ and $\gamma$
compatible with the usual commutativity constraint
$M \otimes_A N \cong N \otimes_A M$ on $A$-modules.
\end{enumerate}
\end{definition}

\noindent
As in the case of the associativity constraint, if such maps
$\sigma_{\mathcal{F}, \mathcal{G}}$ exist, then they are unique.

\section{Direct products}
\label{section-direct-products}

\begin{definition}
\label{definition-direct-products}
Let $(R \to A, \mathcal{C}, \otimes, F, \gamma, \mathcal{O}, \mu)$
be as in Situation \ref{situation-category}. We introduce the following
axiom:
\begin{enumerate}
\item[(I)] The category $\mathcal{C}$ has arbitrary direct
products, and $F$ commutes with them.
\end{enumerate}
\end{definition}

\noindent
If this is the case, then the category $\mathcal{C}$ has inverse limits
and the functor $F$ commutes with them, which is why we use the letter (I)
to indicate this axiom.

\medskip\noindent
In the following lemma and its proof we will use the following
abuse of notation. Suppose that $\mathcal{F}$, $\mathcal{G}$ are
two objects of $\mathcal{C}$, and that
$\alpha : F(\mathcal{F}) \to F(\mathcal{G})$ is an $A$-module map.
We say that $\alpha$ is a {\it morphism of $\mathcal{C}$} if
there exists a morphism $a : \mathcal{F} \to \mathcal{G}$ in $\mathcal{C}$
such that $F(a) = \alpha$. Note that if $a$ exists it is unique.

\begin{lemma}
\label{lemma-quotient}
Assume we are in Situation \ref{situation-category} and that (I) holds.
Let $\mathcal{F}$, $\mathcal{G}$ be two objects of $\mathcal{C}$.
Let $\alpha : F(\mathcal{F}) \to F(\mathcal{G})$ be an $A$-module
map. The functor
$$
\mathcal{C} \longrightarrow \textit{Sets},\quad
\mathcal{H} \longmapsto
\{\varphi \in \text{Mor}_\mathcal{C}(\mathcal{G}, \mathcal{H})
\mid
F(\varphi) \circ \alpha\text{ is a morphism of }\mathcal{C}\}
$$
is representable. The universal object $\mathcal{G} \to \mathcal{G}'$
is a surjection.
\end{lemma}

\begin{proof}
Since $\mathcal{C}$ is abelian, any morphism
$\pi : \mathcal{G} \to \mathcal{H}$ factors uniquely as
$\mathcal{G} \to \mathcal{H}' \to \mathcal{H}$ where
the first map $\pi'$ is a surjection and the second is
an injection. If $F(\pi) \circ \alpha = F(a)$ is a morphism of
$\mathcal{C}$, then $a$ factors through $\mathcal{H}'$ and
we see that $F(\pi') \circ \alpha$ is a morphism of $\mathcal{C}$.
Hence it suffices to consider surjections. Consider the set
$T = \{\pi : \mathcal{G} \to \mathcal{H}_\pi\}$ of surjections
$\pi$ such that $F(\pi) \circ \alpha$ is a morphism of
$\mathcal{C}$. Set
$$
\mathcal{G}'
=
\text{Im}(\mathcal{G} \longrightarrow
\prod\nolimits_{\pi \in T} \mathcal{H}_\pi).
$$
The rest is clear.
\end{proof}

\section{Symmetric products}
\noindent
We introduce the axiom (S) and show that either axiom (I) or (C) implies (S).

\begin{definition}
\label{definition-internal-symmetric}
Let $(R \to A, \mathcal{C}, \otimes, F, \gamma, \mathcal{O}, \mu)$
be as in Situation \ref{situation-category}. We introduce the following
axiom:
\begin{enumerate}
\item[(S)] For every object $\mathcal{F}$ of $\mathcal{C}$ and any
$n \geq 1$ there exists a quotient
$$
\mathcal{F}^{\otimes n} \longrightarrow \text{Sym}^n_{\mathcal{C}}(\mathcal{F})
$$
such that the map of $A$-modules
$
F(\mathcal{F}^{\otimes n})
\longrightarrow 
F(\text{Sym}^n_{\mathcal{C}}(\mathcal{F}))
$
factors through the natural surjection
$F(\mathcal{F})^{\otimes n} \to \text{Sym}^n_A(F(\mathcal{F}))$, and
such that $\text{Sym}^n_{\mathcal{C}}(\mathcal{F})$ is universal with this
property.
\end{enumerate}
\end{definition}

\noindent
Note that if axiom (S) holds, then the universality implies the rule
$\mathcal{F} \leadsto \text{Sym}^n_{\mathcal{C}}(\mathcal{F})$ is a functor.
Moreover, for every $n, m \geq 0$ there are canonical maps
$$
\text{Sym}^n_{\mathcal{C}}(\mathcal{F})
\otimes
\text{Sym}^m_{\mathcal{C}}(\mathcal{F})
\longrightarrow
\text{Sym}^{n + m}_{\mathcal{C}}(\mathcal{F}).
$$
If axiom (D) holds as well, then this will turn
$\bigoplus_{n \geq 0} \text{Sym}^n_{\mathcal{C}}(\mathcal{F})$ into
a weakly commutative ring object of $\mathcal{C}$ (see Definitions
\ref{definition-ring} and \ref{definition-weakly-commutative-ring} below).

\begin{lemma}
\label{lemma-symmetric}
In Situation \ref{situation-category}, if either axiom (C) or (I) holds, then axiom (S) holds.
\end{lemma}

\begin{proof}
Suppose (C) holds.  If $\mathcal{F}$ is an object of $\mathcal{C}$, using the 
maps $\sigma_{\mathcal{F}, \mathcal{F}}$ we get an action of
the symmetric group $S_n$ on $n$ letters on $\mathcal{F}^{\otimes n}$
(to see that it is an action of $S_n$ apply the faithful functor $F$).
Thus, $\text{Sym}_{\mathcal{C}}^n(\mathcal{F})$ can be defined as the cokernel of
a map
$$
\bigoplus\nolimits_{\tau \in S_n} \mathcal{F}^{\otimes n}
\longrightarrow
\mathcal{F}^{\otimes n}
$$
where in the summand corresponding to $\tau$ we use the difference
of the identity and the map corresponding to $\tau$.

\medskip \noindent
Suppose (I) holds.  Let $\mathcal{F}$ be an object of $\mathcal{C}$. The quotient
$\mathcal{F}^{\otimes n} \to \text{Sym}^n_{\mathcal{C}}(\mathcal{F})$
is characterized by the property that if
$a : \mathcal{F}^{\otimes n} \to \mathcal{G}$ is a map such that $F(a)$
factors through
$F(\mathcal{F})^{\otimes n} \to \text{Sym}^n_A(F(\mathcal{F}))$
then $a$ factors in $\mathcal{C}$ through the map to
$\text{Sym}^n_{\mathcal{C}}(\mathcal{F})$. To prove such a quotient exists
apply Lemma \ref{lemma-quotient} to the map
$$
\bigoplus\nolimits_{\tau \in S_n} F(\mathcal{F})^{\otimes n}
\longrightarrow
F(\mathcal{F})^{\otimes n}
$$
mentioned above.
\end{proof}

\section{Ring objects}
\label{section-ring}

\begin{definition}
\label{definition-ring}
Let $(R \to A, \mathcal{C}, \otimes, F, \gamma, \mathcal{O}, \mu)$
be as in Situation \ref{situation-category}.
\begin{enumerate}
\item A {\it ring object} $\mathcal{A}$ in $\mathcal{C}$ consists
of an object $\mathcal{A}$ of $\mathcal{C}$ endowed with maps
$\mathcal{O} \to \mathcal{A}$ and
$\mu_\mathcal{A} : \mathcal{A} \otimes \mathcal{A} \to \mathcal{A}$
which on applying $F$ induce an $A$-algebra structure on $F(\mathcal{A})$.
\item If $\mathcal{A}$ is a ring object of $\mathcal{C}$, then a
{\it (left) module object} over $\mathcal{A}$ is an object $\mathcal{F}$
endowed with a morphism
$\mu_\mathcal{F} : \mathcal{A} \otimes \mathcal{F} \to \mathcal{F}$
such that $F(\mathcal{A}) \otimes_A F(\mathcal{F}) \to F(\mathcal{F})$
induces an $F(\mathcal{A})$-module structure on $F(\mathcal{F})$.
\end{enumerate}
\end{definition}

\noindent
If $\mathcal{A}$ is a ring object of $\mathcal{C}$, then $\Gamma(\mathcal{A})$
inherits an $R$-algebra structure in a natural manner. In other words,
we have the following diagram of rings
$$
\xymatrix{
R \ar[rr] \ar[d] & & A \ar[d] \\
\Gamma(\mathcal{O}) \ar[r] &
\Gamma(\mathcal{A}) \ar[r] &
F(\mathcal{A})
}
$$
In the same vein, given a $\mathcal{A}$-module $\mathcal{F}$ the global
sections $\Gamma(\mathcal{F})$ are a $\Gamma(\mathcal{A})$-module in a
natural way.
Let $\text{Mod}_\mathcal{A}$ denote the category of $\mathcal{A}$-modules.

\begin{lemma} 
\label{lemma-module-objects-abelian}
Let $(R \to A, \mathcal{C}, \otimes, F, \gamma, \mathcal{O}, \mu)$
be as in Situation \ref{situation-category}.
If $\mathcal{A}$ is a ring object in $\mathcal{C}$, then the category 
$\text{Mod}_{\mathcal{A}}$ is abelian.
\end{lemma}

\begin{proof}
Let $\varphi : \mathcal{F} \to \mathcal{G}$ be a map of $\mathcal{A}$-modules.
Set $\mathcal{K} = \text{Ker}(\varphi)$ and
$\mathcal{Q} = \text{Coker}(\varphi)$ in $\mathcal{C}$.
We claim that both $\mathcal{K}$
and $\mathcal{Q}$ have natural $\mathcal{A}$-module structure that turn
them into the kernel and cokernel of $\varphi$ in $\text{Mod}_\mathcal{A}$.
To see this for $\mathcal{K}$ consider the map
$$
\mathcal{A} \otimes \mathcal{K} \to
\mathcal{A} \otimes \mathcal{F} \to
\mathcal{F}
$$
Its composition with the map to $\mathcal{G}$ is zero as $\varphi$ is a
map of $\mathcal{A}$-modules. Hence we see that it factors into a map
$\mathcal{A} \otimes \mathcal{K} \to \mathcal{K}$. To get the module structure
for $\mathcal{Q}$, note that the sequence
$$
\mathcal{A} \otimes \mathcal{F} \to
\mathcal{A} \otimes \mathcal{G} \to
\mathcal{A} \otimes \mathcal{Q} \to 0
$$
is exact, because it is exact on applying $F$. Hence the module structure
on $\mathcal{G}$ induces one on $\mathcal{Q}$. We omit checking that these
structures do indeed give the kernel and cokernel of $\varphi$
in $\text{Mod}_\mathcal{A}$.
\end{proof}

\noindent
Let us use $\text{Hom}_\mathcal{A}(-, -)$ for the morphisms in the 
category $\text{Mod}_{\mathcal{A}}$. Note that
$$
\Gamma(\mathcal{F})
=
\text{Mor}_\mathcal{C}(\mathcal{O}, \mathcal{F})
=
\text{Hom}_\mathcal{A}(\mathcal{A}, \mathcal{F})
$$
for $\mathcal{F} \in \text{Mod}_\mathcal{A}$. The map from the left to
the right associates to $f : \mathcal{O} \to \mathcal{F}$ the map
$$
\mathcal{A} = \mathcal{A} \otimes \mathcal{O}
\xrightarrow{1 \otimes f} \mathcal{A} \otimes \mathcal{F}
\xrightarrow{\mu_\mathcal{F}} \mathcal{F}.
$$

\begin{lemma}
\label{lemma-adjoint-ring}
In Situation \ref{situation-category} assume axiom (D) and let
$\mathcal{A}$ be a ring object in $\mathcal{C}$. Then the functor
$$
\Gamma :
\text{Mod}_\mathcal{A}
\longrightarrow
\text{Mod}_{\Gamma(\mathcal{A})}
$$
has a right adjoint
$$
\mathcal{A} \otimes_{\Gamma(\mathcal{A})} - :
\text{Mod}_{\Gamma(\mathcal{A})}
\longrightarrow
\text{Mod}_\mathcal{A}.
$$
We have
$\mathcal{A} \otimes_{\Gamma(\mathcal{A})} \Gamma(\mathcal{A}) = \mathcal{A}$
and
$F(\mathcal{A} \otimes_{\Gamma(\mathcal{A})} M) =
F(\mathcal{A}) \otimes_{\Gamma(\mathcal{A})} M$.
\end{lemma}

\begin{proof}
The proof is identical to the argument of Lemma \ref{lemma-adjoint} using that
$\Gamma(\mathcal{F}) = \text{Hom}_\mathcal{A}(\mathcal{A}, \mathcal{F})$
for any $\mathcal{A}$-module $\mathcal{F}$.
\end{proof}

\begin{remark}
\label{remark-polynomial-rings}
Let $(R \to A, \mathcal{C}, \otimes, F, \gamma, \mathcal{O}, \mu)$
be as in Situation \ref{situation-category}. Assume axiom (D).
Let $\mathcal{A}$ be a ring object, and let $S$ be a set.
We can define the {\it polynomial algebra over $\mathcal{A}$}
as the ring object
$$
\mathcal{A}[x_s; s \in S] = 
\mathcal{A} \otimes_{\Gamma(\mathcal{A})} (\Gamma(\mathcal{A})[x_s; s \in S])
$$
Explicitly $\mathcal{A}[x_s; s \in S] =  \bigoplus\nolimits_I \mathcal{A}x^I$
where $I$ runs over all functions $I : S \to \mathbf{Z}_{\geq 0}$
with finite support. The symbol $x^I = \prod x_s^{I(s)}$ indicates
the corresponding monomial. The multiplication on
$\mathcal{A}[x_s; s \in S]$ is defined by requiring the ``elements''
of $\mathcal{A}$ to commute with the variables $x_s$. 

A homomorphism
$\mathcal{A}[x_s; s \in S] \to \mathcal{B}$ of ring objects is
given by a homomorphism $\mathcal{A} \to \mathcal{B}$ of ring objects
together with some elements $y_s \in \Gamma(\mathcal{B})$ which commute with
all elements in the image of $F(\mathcal{A}) \to F(\mathcal{B})$.
%Some details omitted.
\end{remark}

\section{Commutative ring objects and modules}
\label{section-commutative-ring}

\noindent
%Here is the naive definition.

\begin{definition}
\label{definition-weakly-commutative-ring}
Let $(R \to A, \mathcal{C}, \otimes, F, \gamma, \mathcal{O}, \mu)$
be as in Situation \ref{situation-category}.
A ring object $\mathcal{A}$ is called {\it weakly commutative} if
$F(\mathcal{A})$ is commutative.
\end{definition}

\begin{lemma}
\label{lemma-quotient-weakly-commutative}
In Situation \ref{situation-category}.
If $\mathcal{A}$ is a weakly commutative ring
and $\mathcal{I} \subset \mathcal{A}$ is a left ideal,
then $\mathcal{I}$ is a two-sided ideal and $\mathcal{A}/\mathcal{I}$
is a weakly commutative ring.
\end{lemma}

\begin{proof}
Consider the image $\mathcal{I}'$ of the multiplication
$\mathcal{A} \otimes \mathcal{I} \to \mathcal{A}$.
By assumption $F(\mathcal{I}') = F(\mathcal{I})$, hence
we have equality. The final assertion is clear.
\end{proof}

\noindent
In order to define the tensor product of two modules over a ring
object we use the notion of commutative modules.

\begin{definition}
\label{definition-commutative-ring}
Let $(R \to A, \mathcal{C}, \otimes, F, \gamma, \mathcal{O}, \mu)$
be as in Situation \ref{situation-category}.
\begin{enumerate}
\item A ring object $\mathcal{A}$ is called {\it commutative} if
there exists an isomorphism
$\sigma : \mathcal{A} \otimes \mathcal{A} \to \mathcal{A} \otimes \mathcal{A}$
which under $F$ gives the usual flip isomorphism and which is compatible
with the multiplication (so in particular $\mathcal{A}$ is weakly
commutative).
\item A module object $\mathcal{F}$ over a ring object
$\mathcal{A}$ is said to be {\it commutative} if there exists an isomorphism
$\sigma : \mathcal{F} \otimes \mathcal{A} \to \mathcal{A} \otimes \mathcal{F}$
which on applying $F$ gives the usual flip isomorphism.
\end{enumerate}
\end{definition}

\noindent
It is clear that if axiom (C) holds, then any weakly commutative ring
object is commutative and all module objects are
automatically commutative.
Let us denote $\text{Mod}^c_\mathcal{A}$
the category of all commutative $\mathcal{A}$-modules. This category
always has cokernels, but not necessarily kernels.

\begin{lemma}
\label{lemma-exact}
Let $(R \to A, \mathcal{C}, \otimes, F, \gamma, \mathcal{O}, \mu)$
be as in Situation \ref{situation-category}. Let $\mathcal{A}$ be a
commutative ring object of $\mathcal{C}$. The category
$\text{Mod}^c_\mathcal{A}$ is abelian in each of the following cases:
\begin{enumerate}
\item axiom (C) holds, or
\item the ring map
$F(\mathcal{A}) \to F(\mathcal{A}) \otimes_A F(\mathcal{A})$
is flat.
\end{enumerate}
The second condition holds for example if $A \to F(\mathcal{A})$ is
either flat or surjective.
\end{lemma}

\begin{proof}
In case (1) we have $\text{Mod}_\mathcal{A} = \text{Mod}^c_\mathcal{A}$ 
so the statement follows from Lemma \ref{lemma-module-objects-abelian}.  
For case (2), let $\varphi : \mathcal{F} \to \mathcal{G}$ be a map
of commutative $\mathcal{A}$-modules.  We set
$\mathcal{K} = \text{Ker}(\varphi)$ and
$\mathcal{Q} = \text{Coker}(\varphi)$ in $\mathcal{C}$, and we know
that these are kernels and cokernels in $\text{Mod}_\mathcal{A}$.
The diagram with exact rows 
$$
\xymatrix{
\mathcal{F} \otimes \mathcal{A} \ar[r] \ar[d]^\sigma &
\mathcal{G} \otimes \mathcal{A} \ar[r] \ar[d]^\sigma &
\mathcal{Q} \otimes \mathcal{A} \ar[r] \ar@{..>}[d] &
0 \\
\mathcal{A} \otimes \mathcal{F} \ar[r] &
\mathcal{A} \otimes \mathcal{G} \ar[r] &
\mathcal{A} \otimes \mathcal{Q} \ar[r] &
0
}
$$
defines the commutativity map $\sigma$ for $\mathcal{Q}$.
But in general we do not know that the map
$\mathcal{K} \otimes \mathcal{A} \to \mathcal{F} \otimes \mathcal{A}$
is injective. After applying $F$ this becomes the map
$$
F(\mathcal{K}) \otimes_A F(\mathcal{A})
\to
F(\mathcal{F}) \otimes_A F(\mathcal{A})
$$
By our discussion in Section \ref{section-ring} we know that
$B = F(\mathcal{A})$ is a commutative $A$-algebra, and
$F(\mathcal{K}) \subset F(\mathcal{F})$ is an inclusion of
$B$-modules. Note that for a $B$-module $M$ we have
$M \otimes_A B = M \otimes_B (B \otimes_A B)$. Hence the injectivity
of the last displayed map is clear if property (2) holds, and in this
case we get the commutativity restraint for $\mathcal{K}$ also.
\end{proof}

\medskip\noindent
If $\mathcal{A}$ is a commutative ring object of $\mathcal{C}$ and
$\mathcal{F}$, $\mathcal{G}$ are module objects over
$\mathcal{A}$, and $\mathcal{F}$ is commutative then we
define
$$
\mathcal{F} \otimes_{\mathcal{A}} \mathcal{G}
:=
\begin{matrix}
\text{Coequalizer of}\\
\text{going around}\\
\text{both ways}
\end{matrix}
\left(
\vcenter{
\xymatrix{
\mathcal{A} \otimes \mathcal{F} \otimes \mathcal{G}
\ar[r]^{\sigma \otimes 1} \ar[d]^{\mu \otimes 1} &
\mathcal{F} \otimes \mathcal{A} \otimes \mathcal{G} \ar[d]^{1 \otimes \mu} \\
\mathcal{F} \otimes \mathcal{G} \ar[r]^1 &
\mathcal{F} \otimes \mathcal{G}
}
}
\right)
$$
Then it is clear that there is a canonical isomorphism
$$
\gamma_\mathcal{A} : 
F(\mathcal{F}) \otimes_{F(\mathcal{A})} F(\mathcal{G})
\longrightarrow
F(\mathcal{F} \otimes_\mathcal{A} \mathcal{G})
$$
which is functorial in the pair $(\mathcal{F}, \mathcal{G})$.
In particular, it is clear that there are functorial isomorphisms
$$
\mu_\mathcal{A} :
\mathcal{A} \otimes_\mathcal{A} \mathcal{F}
\longrightarrow
\mathcal{F},\quad
\mu_\mathcal{A} :
\mathcal{F} \otimes_\mathcal{A} \mathcal{A}
\longrightarrow
\mathcal{F}
$$
for any commutative $\mathcal{A}$-module $\mathcal{F}$ (via $\sigma$ and
the multiplication map for $\mathcal{F}$).

\begin{lemma}
\label{lemma-recover-sitation} \label{lemma-direct-summands-inherited}
Let $(R \to A, \mathcal{C}, \otimes, F, \gamma, \mathcal{O}, \mu)$
be as in Situation \ref{situation-category}. Let $\mathcal{A}$ be a
commutative ring object of $\mathcal{C}$. Assume the category
$\text{Mod}^c_\mathcal{A}$ is abelian. Then
$$
(R \to F(\mathcal{A}),
\text{Mod}^c_\mathcal{A},
\otimes_\mathcal{A},
F, \gamma_\mathcal{A}, \mathcal{A}, \mu_\mathcal{A})
$$
is another set of data as in Situation \ref{situation-category}. 
 Furthermore, if axiom (D) is satisfied for $(R \to A, \mathcal{C}, 
 \otimes, F, \gamma, \mathcal{O}, \mu)$, then it is also satisfied for 
$
(R \to F(\mathcal{A}),
\text{Mod}^c_\mathcal{A},
\otimes_\mathcal{A},
F, \gamma_\mathcal{A}, \mathcal{A}, \mu_\mathcal{A})
$.
\end{lemma}

\begin{proof}
This is clear from the discussion above.
\end{proof}

\noindent
In the situation of the lemma we have the global sections functor
$$
\Gamma_{\mathcal{A}} : 
\text{Mod}_\mathcal{A}
 \longrightarrow 
 \text{Mod}_R,
\quad
\mathcal{F}
\longmapsto
 \text{Hom}_\mathcal{A}(\mathcal{A}, \mathcal{F}).
$$
We have seen in Section \ref{section-ring} that for an object $\mathcal{F} \in
\text{Mod}_\mathcal{A}$ we have
$\Gamma_{\mathcal{A}}(\mathcal{F}) = \Gamma(\mathcal{F})$ as 
$R$-modules. We will often abuse notation by writing 
$\Gamma = \Gamma_{\mathcal{A}}$.

%\begin{remark}
%\label{remark-bialgebra}
%Can one define a notion of a bialgebra in $\mathcal{C}$ and
%repeat the above constructions in this setting?
%s\end{remark}

\section{Finiteness conditions}
\label{section-finite}

\noindent
Here are some finiteness conditions we can impose.

\begin{definition}
\label{definition-properties}
Let
$(R \to A, \mathcal{C}, \otimes, F, \gamma, \mathcal{O}, \mu)$
be as in Situation \ref{situation-category}.
\begin{enumerate}
\item An object $\mathcal{F}$ of $\mathcal{C}$ is said to be
{\it of finite type} if $F(\mathcal{F})$ is a finitely generated $A$-module.
%\item An object $\mathcal{F}$ of $\mathcal{C}$ is said to be
%{\it of finite presentation} if $F(\mathcal{F})$ is an $A$-module
%of finite presentation.
%\item An object $\mathcal{F}$ of $\mathcal{C}$ if said to be
%{\it flat over $R$} if $F(\mathcal{F})$ is a flat $R$-module.
\item An ring object $\mathcal{A}$ of $\mathcal{C}$ is said to be
{\it of finite type} if $F(\mathcal{A})$ is a finitely generated
$A$-algebra.
%\item An ring object $\mathcal{A}$ of $\mathcal{C}$ is said to be
%{\it of finite presentation} if $F(\mathcal{A})$ is a finitely
%presented $A$-algebra.
\item A module object $\mathcal{F}$ over a ring object $\mathcal{A}$
of $\mathcal{C}$ is said to be {\it of finite type} if
$F(\mathcal{F})$ is of finite type over $F(\mathcal{A})$.
%\item A module object $\mathcal{F}$ over a ring object $\mathcal{A}$
%of $\mathcal{C}$ is said to be {\it of finite presentation} if
%$F(\mathcal{F})$ is a module of finite presentation over $F(\mathcal{A})$.
\end{enumerate}
\end{definition}

\noindent
Note that the ring objects in this definition need not be commutative.
A noncommutative algebra $S$ over $A$ is finitely generated %(resp.\ presented)
if it is isomorphic to a quotient of the free algebra
$A\langle x_1, \ldots, x_n\rangle$ for some $n$%(resp. 
%quotient $A\langle x_1, \ldots, x_n\rangle/I$ for some finitely generated two
%sided ideal $I$)
.

\section{Adequacy}
\label{section-adequacy}

\noindent
The notion of {\it adequacy}, which is our analogue of geometric reductivity,
can be formulated in a variety of different ways.

\begin{definition}
\label{definition-noetherian}
Let
$(R \to A, \mathcal{C}, \otimes, F, \gamma, \mathcal{O}, \mu)$
be as in Situation \ref{situation-category}. We introduce the following
axiom:
\begin{enumerate}
\item[(N)] The ring $A$ is Noetherian.
\end{enumerate}
\end{definition}

\begin{definition}
\label{definition-locally-finite}
Let $(R \to A, \mathcal{C}, \otimes, F, \gamma, \mathcal{O}, \mu)$
be as in Situation \ref{situation-category}. An object $\mathcal{F}$
of $\mathcal{C}$ is called {\it locally finite} if it is a filtered
colimit $\mathcal{F} = \text{colim}\ \mathcal{F}_i$
of finite type objects $\mathcal{F}_i$ such that also
$F(\mathcal{F}) = \text{colim}\ F(\mathcal{F}_i)$.
\end{definition}

\begin{definition}
\label{definition-locally-finite-axiom}
Let
$(R \to A, \mathcal{C}, \otimes, F, \gamma, \mathcal{O}, \mu)$
be as in Situation \ref{situation-category}. We introduce the axiom:
\begin{enumerate}
\item[(L)] Every object $\mathcal{F}$ of $\mathcal{C}$ is locally finite.
\end{enumerate}
\end{definition}

\begin{lemma}
\label{lemma-locally-finite-abelian}
Let $(R \to A, \mathcal{C}, \otimes, F, \gamma, \mathcal{O}, \mu)$
be as in Situation \ref{situation-category}.
A quotient of a locally finite object of $\mathcal{C}$ is locally finite.
If axioms (N) and (D) hold, then a subobject of a locally finite object
is locally finite and the subcategory of locally finite objects is abelian.
\end{lemma}

\begin{proof}
Suppose that $\mathcal{F} \to \mathcal{Q}$
is surjective and that $\mathcal{F}$ is locally finite.
Write $\mathcal{F} = \text{colim}\ \mathcal{F}_i$
of finite type objects $\mathcal{F}_i$ such that also
$F(\mathcal{F}) = \text{colim}\ F(\mathcal{F}_i)$. Set 
$\mathcal{Q}_i = \text{Im}(\mathcal{F}_i \to \mathcal{Q})$.
We claim that $\mathcal{Q} = \text{colim}_i\ \mathcal{Q}_i$
and that $F(\mathcal{Q}) = \text{colim}\ F(\mathcal{Q}_i)$.
The last statement follows from exactness of $F$ and the
fact that colimits commute with images in $\text{Mod}_A$.
If $\beta_i : \mathcal{Q}_i \to \mathcal{G}$ is a compatible system
of maps to an object of $\mathcal{C}$, then composing with
the surjections $\mathcal{F}_i \to \mathcal{Q}_i$ gives
a compatible system of maps also, whence a morphism
$\beta : \mathcal{F} \to \mathcal{G}$. But $F(\beta)$ factors
through $F(\mathcal{F}) \to F(\mathcal{Q})$ and hence is zero
on $F(\text{Ker}(\mathcal{F} \to \mathcal{Q})$. Because $F$
is faithful and exact we see that $\beta$ factors as
$\mathcal{Q} \to \mathcal{G}$ as desired.

\medskip\noindent
Suppose that $\mathcal{J} \to \mathcal{F}$
is injective, that $\mathcal{F}$ is locally finite and that (N) and (D) hold.
Write $\mathcal{F} = \text{colim}\ \mathcal{F}_i$
of finite type objects $\mathcal{F}_i$ such that also
$F(\mathcal{F}) = \text{colim}\ F(\mathcal{F}_i)$. By the argument
of the preceding paragraph applied to
$\text{id}_\mathcal{F} : \mathcal{F} \to \mathcal{F}$ we may
assume $\mathcal{F}_i \subset \mathcal{F}_i$ for each $i$.
Set $\mathcal{J}_i = \mathcal{F}_i \cap \mathcal{J}$.
Since axiom (N) holds we see that each $\mathcal{J}_i$ is of finite type.
As $F$ is exact we see that
$\text{colim}\ F(\mathcal{J}_i) = F(\mathcal{J})$.
As axiom (D) holds we know that
$\mathcal{J}' = \text{colim} \mathcal{J}_i$ exists
and $\text{colim}\ F(\mathcal{J}_i) = F(\mathcal{J}')$.
Hence we get a canonical map $\mathcal{J}' \to \mathcal{J}$
which has to be an isomorphism as $F$ is exact and faithful.
This proves that $\mathcal{J}$ is locally finite.

\medskip\noindent
Assume (N) and (D).
Let $\alpha : \mathcal{F} \to \mathcal{G}$
be a morphism of locally finite objects.
We have to show that the kernel and cokernel of $\alpha$ are
locally finite. This is clear by the results of the preceding
two paragraphs.
\end{proof}

\begin{lemma}
\label{lemma-tensor-locally-finite}
Let $(R \to A, \mathcal{C}, \otimes, F, \gamma, \mathcal{O}, \mu)$
be as in Situation \ref{situation-category}. Assume axiom (D) holds.
The tensor product of locally finite objects is locally finite.
For any $R$-module $M$ the object $M \otimes_R \mathcal{O}$ is locally finite.
If $\mathcal{A}$ is a locally finite ring object, then
$\mathcal{A} \otimes_{\Gamma(\mathcal{A})} M$ is locally finite
for any $\Gamma(\mathcal{A})$-module $M$.
\end{lemma}

\begin{proof}
This is clear since in the presence of (D), the tensor product commutes
with colimits.
\end{proof}

\begin{lemma}
\label{lemma-adequate-equivalences}
Let $(R \to A, \mathcal{C}, \otimes, F, \gamma, \mathcal{O}, \mu)$
be as in Situation \ref{situation-category}. Assume axiom (S)
holds. Consider the following conditions
\begin{enumerate}
\item For every surjection of finite type objects
$\mathcal{G} \to \mathcal{F}$ and $f \in \Gamma(\mathcal{F})$ there
exists an $n > 0$ and a $g \in \Gamma(\text{Sym}^n_\mathcal{C}(\mathcal{G}))$
which maps to $f^n$ in $\Gamma(\text{Sym}^n_\mathcal{C}(\mathcal{F}))$.
\item For every surjection $\mathcal{G} \to \mathcal{O}$ with
$\mathcal{G}$ of finite type and $f \in \Gamma(\mathcal{O})$ there
exists an $n > 0$ and a $g \in \Gamma(\text{Sym}^n_\mathcal{C}(\mathcal{G}))$
which maps to $f^n$ in $\Gamma(\mathcal{O})$.
\item For every surjection of weakly commutative ring objects
$\mathcal{A} \to \mathcal{B}$ in $\mathcal{C}$ with $\mathcal{A}$ locally
finite, and any
$f \in \Gamma(\mathcal{B})$, there exists an $n > 0$ and an element
$g \in \Gamma(\mathcal{A})$ such that $g \mapsto f^n$ in $\Gamma(\mathcal{B})$.
\end{enumerate}
We always have $(1) \Rightarrow (2)$ and $(1) \Rightarrow (3)$.
If axiom (N) holds, then $(2) \Rightarrow (1)$.
If axiom (D) holds, then $(3) \Rightarrow (1)$.  Furthermore, consider the following variations
\begin{enumerate}
\item[(1')] For every surjection of objects
$\mathcal{G} \to \mathcal{F}$ and $f \in \Gamma(\mathcal{F})$ there
exists an $n > 0$ and a $g \in \Gamma(\text{Sym}^n_\mathcal{C}(\mathcal{G}))$
which maps to $f^n$ in $\Gamma(\text{Sym}^n_\mathcal{C}(\mathcal{F}))$.
\item[(2')] For every surjection
 $\mathcal{G} \to \mathcal{O}$ and $f \in \Gamma(\mathcal{O})$ there
exists an $n > 0$ and a $g \in \Gamma(\text{Sym}^n_\mathcal{C}(\mathcal{G}))$
which maps to $f^n$ in $\Gamma(\mathcal{O})$.
\item[(3')] For every surjection of weakly commutative ring objects
$\mathcal{A} \to \mathcal{B}$ in $\mathcal{C}$, and any
$f \in \Gamma(\mathcal{B})$, there exists an $n > 0$ and an element
$g \in \Gamma(\mathcal{A})$ such that $g \mapsto f^n$ in $\Gamma(\mathcal{B})$.
\end{enumerate}
If axiom (L) holds, then $(1) \Leftrightarrow (1')$, 
$(2) \Leftrightarrow (2')$, and $(3) \Leftrightarrow (3')$.
\end{lemma}

\begin{proof}
It is clear that (1) implies (2). Assume (N) $+$ (2) and let us prove (1).
Consider $\mathcal{G} \to \mathcal{F}$ and
$f$ as in (1). Let $\mathcal{H} = \mathcal{G} \times_\mathcal{F} \mathcal{O}$.
Then $\mathcal{H} \to \mathcal{O}$ is surjective, and
$F(\mathcal{H}) = F(\mathcal{G}) \times_{F(\mathcal{F})} A$. By assumption
(N) this implies that $F(\mathcal{H})$ is a finite $A$-module.

\medskip\noindent
Let us prove that (1) implies (3). Let $\mathcal{A} \to \mathcal{B}$ and
$f$ be as in (3). Write $\mathcal{A} = \text{colim}_i\ \mathcal{G}_i$ as
a directed colimit such that
$F(\mathcal{A}) = \text{colim}_i\ F(\mathcal{G}_i)$ and
such that each $\mathcal{G}_i$ is of finite type. Think of
$f \in \Gamma(\mathcal{B}) \subset F(\mathcal{B})$. Then for some $i$
there exists a $\tilde f \in F(\mathcal{G}_i)$ which maps to $f$.
Set $\mathcal{G} = \mathcal{G}_i$, set
$\mathcal{F} = \text{Im}(\mathcal{G}_i \to \mathcal{B})$. The map
$\mathcal{G} \to \mathcal{F}$ is surjective. Since $F$ is exact we
see that $f \in F(\mathcal{F}) \subset F(\mathcal{B})$. Hence, as
$\Gamma$ is left exact we conclude that $f \in \Gamma(\mathcal{F})$ as
well. Thus property (1) applies and we find an $n > 0$ and a
$g \in \Gamma(\text{Sym}^n_\mathcal{C}(\mathcal{G}))$
which maps to $f^n$ in $\Gamma(\text{Sym}^n_\mathcal{C}(\mathcal{F}))$.
Since $\mathcal{A}$ and $\mathcal{B}$ are ring objects we obtain
a canonical diagram
$$
\xymatrix{
\mathcal{G}^{\otimes n} \ar[r] \ar[d] &
\mathcal{F}^{\otimes n} \ar[d] \\
\mathcal{A} \ar[r] & \mathcal{B}
}
$$
Since $\mathcal{A}$ and $\mathcal{B}$ are weakly commutative this
produces a commutative diagram
$$
\xymatrix{
\text{Sym}^n_\mathcal{C}(\mathcal{G}) \ar[r] \ar[d] &
\text{Sym}^n_\mathcal{C}(\mathcal{F}) \ar[d] \\
\mathcal{A} \ar[r] & \mathcal{B}
}
$$
Hence the element $g \in \Gamma(\text{Sym}^n_\mathcal{C}(\mathcal{G}))$
maps to the desired element of $\Gamma(\mathcal{A})$.

\medskip\noindent
If (D) holds, then given $\mathcal{G} \to \mathcal{F}$ as in (1)
we can form the map of ``symmetric'' algebras
$$
\text{Sym}^*_\mathcal{C}(\mathcal{G})
\longrightarrow
\text{Sym}^*_\mathcal{C}(\mathcal{F})
$$
and we see that (3) implies (1).

\medskip \noindent
The final statement is clear.
\end{proof}

\noindent
We do not know of an example of Situation \ref{situation-category} where
axiom (D) does not hold. On the other hand, we do know cases where (S)
does not hold, namely, the category of comodules over a general bialgebra.
Hence we take property (3) of the lemma above as the defining property,
since it also make sense in those situations.

\begin{definition}
\label{definition-adequate}
Let
$(R \to A, \mathcal{C}, \otimes, F, \gamma, \mathcal{O}, \mu)$
be as in Situation \ref{situation-category}. We introduce the following
axiom:
\begin{enumerate}
\item[(A)] For every surjection of weakly commutative rings
$\mathcal{A} \to \mathcal{B}$ in $\mathcal{C}$ with $\mathcal{A}$ locally
finite, and any
$f \in \Gamma(\mathcal{B})$, there exists an $n > 0$ and an element
$g \in \Gamma(\mathcal{A})$ such that $g \mapsto f^n$ in $\Gamma(\mathcal{B})$.
\end{enumerate}
\end{definition}

\noindent
A much stronger condition is the notion of {\it goodness},
which is our analogue of linear reductivity. It
can hold even in geometrically interesting situations.

\begin{definition}
\label{definition-good}
Let
$(R \to A, \mathcal{C}, \otimes, F, \gamma, \mathcal{O}, \mu)$
be as in Situation \ref{situation-category}. We introduce the following
axiom:
\begin{enumerate}
\item[(G)] The functor $\Gamma$ is exact.
\end{enumerate}
\end{definition}

\section{Preliminary results}
\label{section-preliminary}

\noindent
Let $\mathcal{A}$ be a weakly commutative ring object of $\mathcal{C}$.
This implies that $\Gamma(\mathcal{A}) \subset F(\mathcal{A})$
is a commutative ring. Let $I \subset \Gamma(\mathcal{A})$ be an ideal. 
Assuming the axiom (D) we have the object
$\mathcal{A} \otimes_{\Gamma(\mathcal{A})} I$
(see Lemma \ref{lemma-adjoint-ring}) and a canonical map
\begin{equation}
\label{equation-multiply-ideal}
\mathcal{A} \otimes_{\Gamma(\mathcal{A})} I
\longrightarrow
\mathcal{A}.
\end{equation}
Namely, this is the adjoint to the map $I \to \Gamma(\mathcal{A})$.
Applying $F$ to the the map (\ref{equation-multiply-ideal}) gives
the obvious map
$F(\mathcal{A}) \otimes_{\Gamma(\mathcal{A})} I \to F(\mathcal{A})$.
The image of (\ref{equation-multiply-ideal}) will be denoted
$\mathcal{A}I$ in the sequel. We have $F(\mathcal{A}I) = F(\mathcal{A})I$
by exactness of the functor $F$.

\medskip\noindent
For an ideal $I$ of a commutative ring $B$ we set
$$
I^* = \{f \in B \mid \exists n > 0,\ f^n \in I^n \}.
$$
Note that it is not clear (or even true) in general that $I^*$
is an ideal. (Our notation is not compatible with notation concerning
integral closure of ideals in algebra texts. We will only use
this notation in this section.)

\begin{lemma}
\label{lemma-adequate}
Assume that we are in Situation \ref{situation-category} and that
axiom (D) holds.
Let $\mathcal{A}$ be a locally finite, weakly commutative ring
object of $\mathcal{C}$.
Let $I \subset \Gamma(\mathcal{A})$ be an ideal.
Consider the ring map
$$
\varphi :
\Gamma(\mathcal{A})/I
\longrightarrow
\Gamma(\mathcal{A}/\mathcal{A}I).
$$
\begin{enumerate}
\item If the axiom (G) holds, $\varphi$ is an isomorphism.
\item If the axiom (A) holds, then
\begin{enumerate}
\item the kernel of $\varphi$ is contained in $I^*\Gamma(\mathcal{A})/I$;
in particular it is locally nilpotent, and
\item for every element $f \in \Gamma(\mathcal{A}/\mathcal{A}I)$
there exists an integer $n > 0$ and an element $g \in \Gamma(\mathcal{A})/I$
which maps to $f^n$ via $\varphi$.
\end{enumerate}
\end{enumerate}
\end{lemma}

\begin{proof}
The surjectivity of $\varphi$ in (1) is immediate from axiom (G).
The ring object $\mathcal{A}/\mathcal{A}I$ is
weakly commutative (by Lemma \ref{lemma-quotient-weakly-commutative}).
Hence (2b) is implied by axiom (A).

\medskip\noindent
Suppose that $f \in \Gamma(\mathcal{A})$ maps to zero in
$\Gamma(\mathcal{A}/\mathcal{A}I)$. This means that
$f \in \Gamma(\mathcal{A}I)$. Choose generators
$f_s \in I$, $s \in S$ for $I$. Consider the ring map
$$
\mathcal{A}[x_s; s \in S]
\longrightarrow
\mathcal{B} = \bigoplus I^n\mathcal{A}
$$
which maps $x_s$ to $f_s \in \Gamma(I\mathcal{A})$, see
Remark \ref{remark-polynomial-rings}.
This is a surjection of ring objects of $\mathcal{C}$.
Hence if (G) holds, then we see that $f$ is in the image
of $\bigoplus_{s \in S} \Gamma(\mathcal{A}) \to \Gamma(\mathcal{A}I)$,
i.e., $f$ is in $\Gamma(\mathcal{A})I$ and injectivity in (1) holds.
For the rest of the proof assume (A). Clearly the polynomial algebra
$\mathcal{A}[x_s; s \in S]$ is weakly commutative and locally finite.
Hence (A) implies there exists an $n > 0$ and an element
$$
g \in \Gamma(\mathcal{A}[x_s; s \in S])
$$
which maps to $f^n$ in the summand $\Gamma(\mathcal{A}I^n)$ of
$\Gamma(\mathcal{B})$. Hence we may also assume that $g$ is in
the degree $n$ summand
$$
\Gamma(\bigoplus\nolimits_{|J| = n} \mathcal{A}x^J)
$$
of $\Gamma(\mathcal{A}[x_s; s \in S])$.
Now, note that there is a ring map $\mathcal{B} \to \mathcal{A}$
and that the composition
$$
\mathcal{A}[x_s; s \in S]
\longrightarrow
\mathcal{B}
\longrightarrow
\mathcal{A}
$$
in degree $n$ maps $\Gamma(\bigoplus_{|J| = n} \mathcal{A}x^J)$ into
$\Gamma(\mathcal{A})I^n$, because $x_s$ maps to $f_s$. Hence
$f^n \in I^n$. This finishes the proof.
\end{proof}

\noindent
Let $\mathcal{A}$ be a weakly commutative ring object of $\mathcal{C}$.
Let $\Gamma(\mathcal{A}) \to \Gamma'$ be a homomorphism of commutative rings.
Write $\Gamma' = \Gamma(\mathcal{A})[x_s; s \in S]/I$.
Assume axiom (D) holds. Then we see that we have the equality
$$
\mathcal{A} \otimes_{\Gamma(\mathcal{A})} \Gamma'
=
\mathcal{A}[x_s; s \in S]/(\mathcal{A}[x_s; s \in S])I
$$
where the polynomial algebra is as in Remark \ref{remark-polynomial-rings}
and the tensor product as in Lemma \ref{lemma-adjoint-ring}.
The reason is that there is an obvious map (from right to left) and
that we have
$$
F(\mathcal{A} \otimes_{\Gamma(\mathcal{A})} \Gamma')
=
F(\mathcal{A})\otimes_{\Gamma(\mathcal{A})} \Gamma'
=
F(\mathcal{A})[x_s; s \in S]/(F(\mathcal{A})[x_s; s \in S])I
$$
by the properties of the functor $F$ and the results mentioned above.
Hence $\mathcal{A} \otimes_{\Gamma(\mathcal{A})} \Gamma'$
is a weakly commutative ring object (see
Lemma \ref{lemma-quotient-weakly-commutative}).
Note that if $\mathcal{A}$ is locally finite, then so is
$\mathcal{A} \otimes_{\Gamma(\mathcal{A})} \Gamma'$, see
Lemma \ref{lemma-tensor-locally-finite}.

\begin{lemma}
\label{lemma-adjoint-new}
Assume that we are in Situation \ref{situation-category} 
and that axiom
(D) holds. Let $\mathcal{A}$ be a ring object.
\begin{enumerate}
\item Assume that also axiom (G) holds. If $M$ is a left
$\Gamma(\mathcal{A})$-module, then the adjunction map
$$
\varphi : 
M 
\longrightarrow 
\Gamma(\mathcal{\mathcal{A}} \otimes_{\Gamma(\mathcal{A})} M)
$$
is an isomorphism.
\item Assume the axiom (A) holds, and that $\mathcal{A}$ is locally
finite and weakly commutative. Let 
$\Gamma(\mathcal{A}) \to \Gamma'$ be a commutative ring map.
Consider the adjunction map
$$
\varphi :
\Gamma'
\longrightarrow
\Gamma(\mathcal{A} \otimes_{\Gamma(\mathcal{A})} \Gamma')
$$
\begin{enumerate}
\item the kernel of $\varphi$ is locally nilpotent, and
\item for every element
$f \in \Gamma(\mathcal{A} \otimes_{\Gamma(\mathcal{A})} \Gamma')$
there exists an integer $n > 0$ and an element
$g \in \Gamma'$ which maps to $f^n$ via $\varphi$.
\end{enumerate}
\end{enumerate}
\end{lemma}

\begin{proof}
For (1), since both functors
$\mathcal{A} \otimes_{\Gamma(\mathcal{A})} -$ 
and $\Gamma$ commute with arbitrary direct sums,
the map $\varphi$ is an  isomorphism when $M$ is free.
Furthermore, since  $\mathcal{A} \otimes_{\Gamma(\mathcal{A})} -$
is right exact and $\Gamma$ is exact, the general case follows.
For (2), the map is an isomorphism when $\Gamma'$ is a polynomial
algebra (since we are assuming all functors commute with direct sums).
And the general case follows from this, the discussion above the
lemma and Lemma \ref{lemma-adequate}.
\end{proof}

\begin{lemma}
\label{lemma-surjective}
Assume that we are in Situation \ref{situation-category} and that
axioms (D) and (A) hold. Then for every locally finite, weakly commutative
ring object $\mathcal{A}$ of $\mathcal{C}$ the map
$$
\text{Spec}(F(\mathcal{A})) \longrightarrow \text{Spec}(\Gamma(\mathcal{A}))
$$
is surjective.
\end{lemma}

\begin{proof}
Let $\Gamma(\mathcal{A}) \to K$ be a ring map to a field.
We have to show that the ring
$$
F(\mathcal{A}) \otimes_{\Gamma(\mathcal{A})} K
=
F(\mathcal{A} \otimes_{\Gamma(\mathcal{A})} K)
$$
is not zero. This follows from Lemma \ref{lemma-adjoint-new}
and the fact that $K$ is not the zero ring.
\end{proof}

\noindent
In the following lemma we use the notion of a \emph{universally subtrusive}
morphism of schemes $f : X \to Y$. This means that $f$
satisfies the following valuation lifting property:
for every valuation ring $V$ and every morphism $\text{Spec}(V) \to Y$ 
there exists a local map of valuation rings $V \to V'$ and a morphism
$\text{Spec}(V') \to X$ such that
$$
\xymatrix{
X \ar[d] & \text{Spec}(V') \ar[l] \ar[d] \\
Y & \text{Spec}(V) \ar[l]
}
$$
is commutative. It turns out that if $f : X \to Y$ is of finite type, and
$Y$ is Noetherian, then this notion is equivalent to $f$ being
\emph{universally submersive}.

\begin{lemma}
\label{lemma-universally-submersive}
Let
$(R \to A, \mathcal{C}, \otimes, F, \gamma, \mathcal{O}, \mu)$
be as in Situation \ref{situation-category}.
Let $\mathcal{A}$ be a ring object.
Assume that
\begin{enumerate}
\item axioms (D) and (A) hold, and
\item $\mathcal{A}$ is locally finite and weakly commutative.
\end{enumerate}
Then $\text{Spec}(F(\mathcal{A})) \to \text{Spec}(\Gamma(\mathcal{A}))$
is universally subtrusive. If in addition, 
\begin{enumerate} \setcounter{enumi}{2}
\item $R \to A$ is finite type,
\item $\mathcal{A}$ is of finite type, and
\item $\Gamma(\mathcal{A})$ is Noetherian.
\end{enumerate}
Then 
$\text{Spec}(F(\mathcal{A})) \to \text{Spec}(\Gamma(\mathcal{A}))$
is universally submersive.
\end{lemma}

\begin{proof}
To show the first part, let 
$\text{Spec}(V) \to \text{Spec}(\Gamma(\mathcal{A}))$ be a 
morphism where $V$ is a valuation ring with fraction field $K$.  
We must show that 
$$
f: \text{Spec} ( F(\mathcal{A}) \otimes_{\Gamma(\mathcal{A})} V)
 \longrightarrow \
 \text{Spec}(V)
 $$
is subtrusive.  Let $\eta \in \text{Spec}(V)$ be the generic point.  It suffices to show
that the closure of $f^{-1}(\eta)$ in 
$\text{Spec} ( F(\mathcal{A}) \otimes_{\Gamma(\mathcal{A})} V)$ 
surjects onto $\text{Spec} (V)$.
If we set 
$$
\mathcal{I} = \text{ker}( \mathcal{A} \otimes_{\Gamma(\mathcal{A})} V 
\longrightarrow 
\mathcal{A} \otimes_{\Gamma(\mathcal{A})} K)
$$ 
then $F(\mathcal{I})$ is the kernel of 
$F(\mathcal{A}) \otimes_{\Gamma(\mathcal{A})} V 
\to F(\mathcal{A}) \otimes_{\Gamma(\mathcal{A})} K$ 
and defines the closure of $f^{-1}(\eta)$.  The ring object 
$(\mathcal{A} \otimes_{\Gamma(\mathcal{A})} V) / \mathcal{I}$ is weakly commutative 
and locally finite.  By Lemma \ref{lemma-surjective}, 
$$
\text{Spec}( F((\mathcal{A} \otimes_{\Gamma(\mathcal{A})} V) / \mathcal{I}) )
\longrightarrow 
\text{Spec}(
\Gamma((\mathcal{A} \otimes_{\Gamma(\mathcal{A})} V) / \mathcal{I})
)
$$
is surjective.  Axiom (A) applied to the surjection 
$\mathcal{A} \otimes_{\Gamma(\mathcal{A})} V 
\to (\mathcal{A} \otimes_{\Gamma(\mathcal{A})} V)/ \mathcal{I}$
implies that
$$
\text{Spec}(
\Gamma((\mathcal{A} \otimes_{\Gamma(\mathcal{A})} V) / \mathcal{I})
)
\longrightarrow 
\text{Spec} (V)
$$
is integral.  Therefore the composition of the two morphisms above is
surjective so that the closure of $f^{-1}(\eta)$ surjects onto
$\text{Spec}(V)$.

\medskip \noindent
The hypotheses in the second part imply that
$\Gamma(\mathcal{A}) \to F(\mathcal{A})$ 
is of finite type and $\Gamma(\mathcal{A})$ is Noetherian,
hence the remark preceding the lemma applies.
\end{proof}

\noindent
Below we will use the following algebraic result to get finite generation.

\begin{theorem}
\label{theorem-finite-type-local}
Consider ring maps $R \to B \to A$ such that
\begin{enumerate}
\item $B$ and $R$ are noetherian,
\item $R \to A$ is of finite type, and
\item $\text{Spec}(A) \to \text{Spec}(B)$ is universally submersive.
\end{enumerate}
Then $R \to B$ is of finite type.
\end{theorem}

\begin{proof}
This is a special case of Theorem 6.2.1 of \cite{alper_adequate}.
It was first discovered while writing an earlier version of this paper.
\end{proof}

\section{The main result}
\label{section-main}

\noindent
The main argument in the proof of Theorem \ref{theorem-finite-type} 
is an induction argument. In order to formulate
it we use the following condition.

\begin{definition}
\label{definition-star}
Let
$(R \to A, \mathcal{C}, \otimes, F, \gamma, \mathcal{O}, \mu)$
be as in Situation \ref{situation-category}. 
Let $\mathcal{A}$ be a weakly
commutative ring object. Consider the following property of $\mathcal{A}$
\begin{enumerate}
\item[($\star$)] The ring $\Gamma(\mathcal{A})$ is a finite type $R$-algebra
and for every finite type module $\mathcal{F}$ over $\mathcal{A}$ the
$\Gamma(\mathcal{A})$-module $\Gamma(\mathcal{F})$ is finite.
\end{enumerate}
\end{definition}

\begin{lemma}
\label{lemma-quotient-star}
Let
$(R \to A, \mathcal{C}, \otimes, F, \gamma, \mathcal{O}, \mu)$
be as in Situation \ref{situation-category}.
Let $\mathcal{A} \to \mathcal{B}$ be a 
surjection of ring objects.
Assume
\begin{enumerate}
\item $R$ is Noetherian and axiom (A) holds,
\item $\mathcal{A}$ is locally finite and weakly commutative, and
\item $\Gamma(\mathcal{B})$ is a finitely generated $R$-algebra.
\end{enumerate}
Then $\Gamma(\mathcal{B})$ is a finite $\Gamma(\mathcal{A})$-module
and there exists a finitely generated $R$-subalgebra
$B \subset \Gamma(\mathcal{A})$ such that
$$
\text{Im}(\Gamma(\mathcal{A}) \longrightarrow \Gamma(\mathcal{B}))
=
\text{Im}(B \longrightarrow \Gamma(\mathcal{B})).
$$
\end{lemma}

\begin{proof}
Since $\mathcal{A}$ is weakly commutative, so is $\mathcal{B}$.
Hence $\Gamma(\mathcal{B})$ is a commutative $R$-algebra.
Pick $f_1, \ldots, f_n \in \Gamma(\mathcal{B})$ which generate
as an $R$-algebra. By axiom (A) we can find
$g_1, \ldots, g_n \in \Gamma(\mathcal{A})$ which map to
$f_1^{n_1}, \ldots, f_n^{n_n}$ in $\Gamma(\mathcal{B})$
for some $n_i > 0$. Then we see that $\Gamma(\mathcal{B})$
is generated by the elements
$$
f_1^{e_1} \ldots f_n^{e_n},\quad 0 \leq e_i \leq n_i - 1
$$
and so $\Gamma(\mathcal{B})$ is finite over $\Gamma(\mathcal{A})$.
As a first approximation, let
$B = R[g_1, \ldots, g_n] \subset \Gamma(\mathcal{A})$.
Then the equality of the lemma may not hold, but in any case
$\Gamma(\mathcal{A})$ is finite over $B$. Since $B$ is a Noetherian
ring, $\text{Im}(\Gamma(\mathcal{A}) \to \Gamma(\mathcal{B}))$ 
is a finite $B$-module so be choose finitely many generators
$g_{n + 1}, \ldots, g_{n + m} \in \Gamma(\mathcal{A})$.
 Hence by setting $B = R[g_1, \ldots, g_{n + m}]$, the lemma is proved.
\end{proof}

\begin{lemma}
\label{lemma-sandwich}
Let
$(R \to A, \mathcal{C}, \otimes, F, \gamma, \mathcal{O}, \mu)$
be as in Situation \ref{situation-category}.
Let $\mathcal{A}$ be a ring
object and let $\mathcal{I} \subset \mathcal{A}$ be a left ideal.
Assume
\begin{enumerate}
\item $R$ is Noetherian and axiom (A) holds,
\item $\mathcal{A}$ is locally finite and weakly commutative,
\item ($\star$) holds for $\mathcal{A}/\mathcal{I}$, and
\item there is a quotient $\mathcal{A} \to \mathcal{A}'$ such that
($\star$) holds for $\mathcal{A}'$ and such that $\mathcal{I}$ is
a finite $\mathcal{A}'$-module.
\end{enumerate}
Then ($\star$) holds for $\mathcal{A}$.
\end{lemma}

\begin{proof}
Since $\mathcal{A}$ is weakly commutative and locally finite so are
$\mathcal{A}/\mathcal{I}$ and $\mathcal{A}'$. By
Lemma \ref{lemma-quotient-star}
the rings $\Gamma(\mathcal{A}')$ and $\Gamma(\mathcal{A}/\mathcal{I})$
are finite $\Gamma(\mathcal{A})$-algebras. Consider the exact sequence
$$
0 \to
\Gamma(\mathcal{I}) \to
\Gamma(\mathcal{A}) \to \Gamma(\mathcal{A}/\mathcal{I}).
$$
By ($\star$) for $\mathcal{A}'$ we see that $\Gamma(\mathcal{I})$ is a finite
$\Gamma(\mathcal{A}')$-module, hence a finite $\Gamma(\mathcal{A})$-module.
Choose generators $x_1, \ldots, x_s \in \Gamma(\mathcal{I})$ as a
$\Gamma(\mathcal{A})$-module. By Lemma \ref{lemma-quotient-star}
we can find a finite type $R$-subalgebra $B \subset \Gamma(\mathcal{A})$
such that the image of $B$ in $\Gamma(\mathcal{A}')$ and the image
of $B$ in $\Gamma(\mathcal{A}/\mathcal{I})$ is the same as the image
of $\Gamma(\mathcal{A})$ in those rings. We claim that
$$
\Gamma(\mathcal{A}) = B[x_1, \ldots, x_s]
$$
as subrings of $\Gamma(\mathcal{A})$. Namely, if $h \in \Gamma(\mathcal{A})$
then we can find an element $b \in B$ which has the same image as $h$
in $\Gamma(\mathcal{A}/\mathcal{I})$. Hence replacing $h$ by $h - b$ we
may assume $h \in \Gamma(\mathcal{I})$. By our choice of
$x_1, \ldots, x_s$ we may write $h = \sum a_i x_i$ for some
$a_i \in \Gamma(\mathcal{A})$. But since $\mathcal{I}$ is a
$\mathcal{A}'$-module, we can write this as
$h = \sum a_i' x_i$ with $a_i' \in \Gamma(\mathcal{A}')$ the image of
$a_i$. By choice of $B$ we can find $b_i \in B$ mapping to $a_i'$.
Hence we see that $h \in B[x_1, \ldots, x_s]$ as desired. This proves that
$\Gamma(\mathcal{A})$ is a finitely generated $R$-algebra.

\medskip\noindent
Let $\mathcal{F}$ be a finite type $\mathcal{A}$-module.
Set $\mathcal{I}\mathcal{F}$ equal to the image of the map
$\mathcal{I} \otimes \mathcal{F} \to \mathcal{F}$ which is the
restriction of the multiplication map of $\mathcal{F}$.
Consider the exact sequence
$$
0 \to
\mathcal{I}\mathcal{F} \to
\mathcal{F} \to
\mathcal{F}/\mathcal{I}\mathcal{F} \to 0
$$
This gives rise to a similar short exact sequence on applying $F$, and a
surjective map
$F(\mathcal{I}) \otimes_A F(\mathcal{F}) \to F(\mathcal{I}\mathcal{F})$
which factors through $F(\mathcal{I}) \otimes_{F(\mathcal{A})} F(\mathcal{F})$
as $\mathcal{A}$ is weakly commutative.
Since $F(\mathcal{F})$ is finite as a $F(\mathcal{A})$-module, and
$F(\mathcal{I})$ is finite as a $F(\mathcal{A}')$-module, we conclude
that $F(\mathcal{I}\mathcal{F})$ is a finite $F(\mathcal{A}')$-module, i.e.,
that $\mathcal{I}\mathcal{F}$ is a finite $\mathcal{A}'$-module.
In the same way we see that $\mathcal{F}/\mathcal{I}\mathcal{F}$
is a finite $\mathcal{A}/\mathcal{I}$-module. Hence in the exact sequence
$$
0 \to
\Gamma(\mathcal{I}\mathcal{F}) \to
\Gamma(\mathcal{F}) \to
\Gamma(\mathcal{F}/\mathcal{I}\mathcal{F})
$$
we see that the modules on the left and the right are finite
$\Gamma(\mathcal{A})$-modules. Since $\Gamma(\mathcal{A})$ is Noetherian
by the result of the preceding paragraph we see that $\Gamma(\mathcal{F})$
is a finite $\Gamma(\mathcal{A})$-module. This conclude the proof that
property ($\star$) holds for $\mathcal{A}$.
\end{proof}

\begin{lemma}
\label{lemma-nilpotent}
Let
$(R \to A, \mathcal{C}, \otimes, F, \gamma, \mathcal{O}, \mu)$
be as in Situation \ref{situation-category}.
Let $\mathcal{A}$ be a ring object,
and let $\mathcal{I} \subset \mathcal{A}$ be a left ideal.
Assume that
\begin{enumerate}
\item axioms (N) and (A) hold and $R$ is Noetherian,
\item $\mathcal{A}$ is locally finite, weakly commutative and of finite type,
\item $\mathcal{I}^n = 0$ for some $n \geq 0$, and
\item $\mathcal{A}/\mathcal{I}$ has property ($\star$).
\end{enumerate}
Then $\mathcal{A}$ has property ($\star$).
\end{lemma}

\begin{proof}
We argue by induction on $n$ and hence we may assume that
$\mathcal{I}^2 = 0$. Then we get an exact sequence
$$
0 \to \mathcal{I} \to \mathcal{A} \to \mathcal{A}/\mathcal{I} \to 0.
$$
Because (N) holds and $\mathcal{A}$ is of finite type we see that
$F(\mathcal{A})$ is a finitely generated $A$-algebra hence Noetherian.
Thus $\mathcal{I}$ is a finite type $\mathcal{A}$-module, and hence also
a finite type $\mathcal{A}/\mathcal{I}$-module. This means that
Lemma \ref{lemma-sandwich} applies, and we win.
\end{proof}

\begin{theorem}
\label{theorem-finite-type}
Let
$(R \to A, \mathcal{C}, \otimes, F, \gamma, \mathcal{O}, \mu)$
be as in Situation \ref{situation-category}.
Assume
\begin{enumerate}
\item $R$ is Noetherian,
\item $R \to A$ is of finite type, and
\item the axioms (A) and (D) hold.
\end{enumerate}
Then for every finite type, locally finite,
weakly commutative ring object $\mathcal{A}$ of $\mathcal{C}$
property ($\star$) holds.
\end{theorem}

\begin{proof}
Let $\mathcal{A}$ be a finite type, locally finite,
weakly commutative ring object $\mathcal{A}$ of $\mathcal{C}$.
For every left ideal $\mathcal{I} \subset \mathcal{A}$ the
quotient $\mathcal{A}/\mathcal{I}$ is also a finite type, locally finite,
weakly commutative ring object of $\mathcal{C}$. Consider the set
$$
\{\mathcal{I} \subset \mathcal{A} \mid 
\text{($\star$) fails for }\mathcal{A}/\mathcal{I} \}.
$$
To get a contradiction assume that this set is nonempty. By Noetherian
induction on the ideal $F(\mathcal{I}) \subset F(\mathcal{A})$ we
see there exists a maximal left ideal $\mathcal{I}_{max} \subset \mathcal{A}$
such that ($\star$) holds for any ideal strictly containing
$\mathcal{I}_{max}$ but ($\star$) does not hold for $\mathcal{I}_{max}$.
Replacing $\mathcal{A}$ by $\mathcal{A}/\mathcal{I}_{max}$
we may assume (in order to get a contradiction)
that ($\star$) does not hold for $\mathcal{A}$ but does
hold for every proper quotient of $\mathcal{A}$.

\medskip\noindent
Let $f \in \Gamma(\mathcal{A})$ be nonzero. If
$\text{Ker}(f : \mathcal{A} \to \mathcal{A})$ is nonzero, then we see that
we get an exact sequence
$$
0 \to (f) \to \mathcal{A} \to \mathcal{A}/(f) \to 0
$$
Since we are assuming ($\star$) holds for both 
$\mathcal{A} / \text{Ker}(f: \mathcal{A} \rightarrow \mathcal{A})$ and $\mathcal{A} / (f)$ 
and since $\text{Ker}(f)$ is a finite $\mathcal{A}/(f)$-module, 
we can apply Lemma \ref{lemma-sandwich}. Hence we see that we may assume
that any nonzero element $f \in \Gamma(\mathcal{A})$ is a nonzero divisor
on $\mathcal{A}$. In particular, $\Gamma(\mathcal{A})$ is a domain.

\medskip\noindent
Again, assume that $f \in \Gamma(\mathcal{A})$ is nonzero.
Consider the sequence
$$
0 \to \mathcal{A} \xrightarrow{f} \mathcal{A} \to
\mathcal{A}/f\mathcal{A} \to 0
$$
which gives rise to the sequence
$$
0 \to \Gamma(\mathcal{A}) \xrightarrow{f} \Gamma(\mathcal{A}) \to
\text{Im}(\Gamma(\mathcal{A}) \to \Gamma(\mathcal{A}/f\mathcal{A})) \to 0
$$
We know that the ring on the right is a finite type $R$-algebra
which is finite over $\Gamma(\mathcal{A})$, see
Lemma \ref{lemma-quotient-star}. Hence any ideal
$I \subset \Gamma(\mathcal{A})$ containing $f$ maps to a finitely generated
ideal in it. This implies that $\Gamma(\mathcal{A})$ is Noetherian.

\medskip\noindent
Next, we claim that for any finite type $\mathcal{A}$-module
$\mathcal{F}$ the module $\Gamma(\mathcal{F})$ is a finite
$\Gamma(\mathcal{A})$-module. Again we can do this by Noetherian
induction applied to the set
$$
\{\mathcal{G} \subset \mathcal{F} \text{ is an }
\mathcal{A}
\text{-submodule such that finite generation fails for }
\Gamma(\mathcal{F}/\mathcal{G}) \}.
$$
In other words, we may assume that $\mathcal{F}$ is a minimal counter example
in the sense that any proper quotient of $\mathcal{F}$ gives a finite
$\Gamma(\mathcal{A})$-module. 
Pick $s \in \Gamma(\mathcal{F})$ nonzero (if $\Gamma(\mathcal{F})$ is zero, we're done).
Let $\mathcal{A}\cdot s \subset \mathcal{F}$ denote the image of
$\mathcal{A} \to \mathcal{F}$ which is multiplying against $s$.
Now we have
$$
0 \to \mathcal{A}\cdot s \to
\mathcal{F} \to
\mathcal{F}/\mathcal{A} \cdot s \to 0
$$
which gives the exact sequence
$$
0 \to \Gamma(\mathcal{A}\cdot s) \to
\Gamma(\mathcal{F}) \to
\Gamma(\mathcal{F}/\mathcal{A} \cdot s)
$$
By minimality we see that the module on the right is finite over
the Noetherian ring $\Gamma(\mathcal{A})$. On the other hand, the module
on the left is $\Gamma(\mathcal{A}/\mathcal{I})$ for the ideal
$\mathcal{I} = \text{Ker}(s: \mathcal{A} \rightarrow \mathcal{F})$. 
If $\mathcal{I} = 0$ then this is
$\Gamma(\mathcal{A})$ and therefore finite, and if $\mathcal{I} \not = 0$ then
this is a finite $\Gamma(\mathcal{A})$-module by
Lemma \ref{lemma-quotient-star} and minimality of $\mathcal{A}$.
Hence we conclude that the middle module is finite over the
Noetherian ring $\Gamma(\mathcal{A})$ which is the desired contradiction.

\medskip\noindent
Finally, we show that $\Gamma(\mathcal{A})$ is of finite type over $R$
which will finish the proof. Namely, by
Lemma \ref{lemma-universally-submersive} the morphism of
schemes
$$
\text{Spec}(F(\mathcal{A})) \longrightarrow \text{Spec}(\Gamma(\mathcal{A}))
$$
is universally submersive. We have already seen that $\Gamma(\mathcal{A})$
is a Noetherian ring. Thus Theorem \ref{theorem-finite-type-local} kicks in and
we are done.
\end{proof}

\begin{remark}  We note that the proof of Theorem \ref{theorem-finite-type} 
can be simplified if the axiom (G) is also satisfied.  In fact, if axiom (G) holds 
in addition to the conditions (1) - (3) of Theorem \ref{theorem-finite-type}, 
then for every finite type, weakly commutative (but not necessarily locally finite) 
ring object $\mathcal{A}$, property ($\star$) holds.  Lemma \ref{lemma-adjoint-new} 
implies that for any ideal $I \subseteq \Gamma(\mathcal{A})$, 
$I = I F(\mathcal{A}) \cap \Gamma(\mathcal{A})$;   
therefore $\Gamma(\mathcal{A})$ is Noetherian.  We can then apply 
 Theorem \ref{theorem-finite-type-local} to conclude that 
 $\Gamma(\mathcal{A})$ is a finite type $R$-algebra.  Furthermore, a simple 
 noetherian induction argument shows that for every finite type module 
 $\mathcal{F}$ over $\mathcal{A}$ the $\Gamma(\mathcal{A})$-module 
 $\Gamma(\mathcal{F})$ is finite type.
\end{remark}

\section{Quasi-coherent sheaves on algebraic stacks}
\label{section-quasi-coherent}

\noindent
Let $S = \text{Spec}(R)$ be an affine scheme.
Let $\mathcal{X}$ be a quasi-compact algebraic stack over $S$.
Let $p : T \to \mathcal{X}$ be a smooth surjective morphism from an
affine scheme $T = \text{Spec}(A)$.

\begin{lemma}
\label{lemma-quasi-coherent}
In the situation above, the category $\text{QCoh}(\mathcal{O}_{\mathcal{X}})$
endowed with its natural tensor product, pullback functor
$F : \text{QCoh}(\mathcal{O}_{\mathcal{X}}) \to
\text{QCoh}(\mathcal{O}_T) = \text{Mod}_A$
and structure sheaf $\mathcal{O} = \mathcal{O}_{\mathcal{X}}$
is an example of Situation \ref{situation-category}.  
The functor $\Gamma : \text{QCoh}(\mathcal{O}_{\mathcal{X}}) \to \text{Mod}_R$
is identified with the functor of global sections
$$ \mathcal{F} \longmapsto \Gamma(\mathcal{X}, \mathcal{F}).$$
Axioms (D), (C), and (S) hold.  If $\mathcal{X}$ is noetherian 
(eg. $\mathcal{X}$ is quasi-separated and $A$ is Noetherian), then axiom (L) holds.
\end{lemma}

\begin{proof}
The final statement is \cite[Prop 15.4]{lmb}.  The rest is clear.
\end{proof}

\noindent
The following definition reinterprets the adequacy axiom (A).

\begin{definition}
\label{definition-adequate-moduli-space}
Let $\mathcal{X}$ be an quasi-compact algebraic stack over $S = \text{Spec}(R)$.  
We say that $\mathcal{X}$ is \emph{adequate} if for every surjection 
$\mathcal{A} \to \mathcal{B}$ of quasi-coherent 
$\mathcal{O}_{\mathcal{X}}$-algebras with $\mathcal{A}$ locally finite and 
$f \in \Gamma(\mathcal{X}, \mathcal{B})$, there
exists an $n > 0$ and a $g \in \Gamma(\mathcal{X}, \mathcal{A})$
such that $g \mapsto f^n$ in $\Gamma(\mathcal{X}, \mathcal{B})$.
\end{definition}

\begin{lemma}
\label{lemma-adequate-stack-equivalences}
Let $\mathcal{X}$ be an quasi-compact algebraic stack over $S = \text{Spec}(R)$. 
The following are equivalent:
\begin{enumerate}
\item $\mathcal{X}$ is adequate.
\item For every surjection of finite type $\mathcal{O}_{\mathcal{X}}$-modules
$\mathcal{G} \to \mathcal{F}$ and $f \in \Gamma(\mathcal{X}, \mathcal{F})$, there
exists an $n > 0$ and a $g \in \Gamma(\mathcal{X}, \text{Sym}^n \mathcal{G})$
such that $g \mapsto f^n$ in $\Gamma(\mathcal{X}, \text{Sym}^n \mathcal{F})$.
\end{enumerate}
If $\mathcal{X}$ is noetherian, then the above are also equivalent to:
\begin{enumerate} \setcounter{enumi}{2}
\item 
For every surjection
$\mathcal{G} \to \mathcal{O}$ with $\mathcal{G}$ of finite type and 
$f \in \Gamma(\mathcal{X}, \mathcal{O}_{\mathcal{X}})$, there
exists an $n > 0$ and a $g \in \Gamma(\mathcal{X}, \text{Sym}^n \mathcal{G})$
such that $g \mapsto f^n$ in $\Gamma(\mathcal{X}, \mathcal{O}_{\mathcal{X}})$.
\item[(1')] 
For every surjection $\mathcal{A} \to \mathcal{B}$ of quasi-coherent 
$\mathcal{O}_{\mathcal{X}}$-algebras and 
$f \in \Gamma(\mathcal{X}, \mathcal{B})$, there
exists an $n > 0$ and a $g \in \Gamma(\mathcal{X}, \mathcal{A})$
such that $g \mapsto f^n$ in $\Gamma(\mathcal{X}, \mathcal{B})$.
\item[(2')]
 For every surjection of $\mathcal{O}_{\mathcal{X}}$-modules
$\mathcal{G} \to \mathcal{F}$ and $f \in \Gamma(\mathcal{X}, \mathcal{F})$, there
exists an $n > 0$ and a $g \in \Gamma(\mathcal{X}, \text{Sym}^n \mathcal{G})$
such that $g \mapsto f^n$ in $\Gamma(\mathcal{X}, \text{Sym}^n \mathcal{F})$.
\item[(3')]
For every surjection
$\mathcal{G} \to \mathcal{O}$  and 
$f \in \Gamma(\mathcal{X}, \mathcal{O}_{\mathcal{X}})$, there
exists an $n > 0$ and a $g \in \Gamma(\mathcal{X}, \text{Sym}^n \mathcal{G})$
such that $g \mapsto f^n$ in $\Gamma(\mathcal{X}, \mathcal{O}_{\mathcal{X}})$.
\end{enumerate}
\end{lemma}

\begin{proof}
This is Lemma \ref{lemma-adequate-equivalences}.
\end{proof}

\begin{corollary}
\label{corollary-stack-finite-type}
Let $\mathcal{X}$ be an algebraic stack finite type over an affine noetherian 
scheme $\text{Spec}(R)$. 
Suppose $\mathcal{X}$ is adequate.  
Let $\mathcal{A}$ be a finite type $\mathcal{O}_{\mathcal{X}}$-algebra.  
Then $\Gamma(\mathcal{X}, \mathcal{A})$ is finitely generated over $R$ 
and for every finite type $\mathcal{A}$-module $\mathcal{\mathcal{F}}$, 
the $\Gamma(\mathcal{X}, \mathcal{A})$-module 
$\Gamma(\mathcal{X}, \mathcal{F})$ is finite.
\end{corollary}

\begin{proof}
This is Theorem \ref{theorem-finite-type}.
\end{proof}

\section{Bialgebras, modules and comodules}
\label{section-hopf}

\noindent
In this section we discuss how modules and comodules over a bialgebra
form an example of our abstract setup.  If $A$ is a commutative ring, 
recall that a \emph{bialgebra $H$ over $A$} is an
$A$-module $H$ endowed with maps
$(A \to H, H \otimes_A H \to A, \epsilon : H \to A,
\delta : H \to H \otimes_A H)$. Here $H \otimes_A H \to H$ and
$A \to H$ define an unital $A$-algebra structure on $H$,
the maps $\delta$ and $\epsilon$ are unital $A$-algebra maps.
Moreover, the comultiplication $\mu$ is associative and
$\epsilon$ is a counit.

\medskip \noindent
Let $H$ be a bialgebra over $A$.  A \emph{left $H$-module} 
is a left module over the $R$-algebra structure on $H$; that is,
 there is a $A$-module homomorphism $H \otimes_A M \to M$ 
satisfying the two commutative diagrams for an action.  A \emph{left $H$-comodule} 
$M$ is an $R$-module homomorphism $\sigma: M \to H \otimes_A M$ satisfying 
the two commutative diagram for a coaction.  See \cite[Chapter 3]{kassel} and 
\cite[Chapter 1]{montgomery} for the basic properties of $H$-modules 
and $H$-comodules.

\begin{definition}
\label{definition-bialgebra}
Let $A$ be a commutative ring.
Let $H$ be a bialgebra over $A$.
\begin{enumerate}
\item Let $\text{Mod}_H$ be the category
of left $H$-modules. It is endowed with the forgetful functor to $A$-modules,
the tensor product
$$
(M, N) \longmapsto M \otimes_A N
$$
where $H$ acts on $M \otimes_A N$ via the comultiplication,
and the object $\mathcal{O}$ given by the
module $A$ where $H$ acts via the counit.
\item Let $\text{Comod}_H$ be the category of left $H$-comodules.
It is endowed with the forgetful functor to $A$-modules, the tensor product
$$
(M, N) \longmapsto M \otimes_A N
$$
where comodule structure on $M \otimes_A N$ comes from the multiplication
in $H$, and the object $\mathcal{O}$ given by the
module $A$ where $H$ acts via the $A$-algebra structure $H$.
\end{enumerate}
\end{definition}

\begin{lemma}
\label{lemma-bialgebra}
Let $R \to A$ be a map of commutative rings.
Let $H$ be a bialgebra over $A$.
\begin{enumerate}
\item The category $\text{Mod}_H$ with its additional structure introduced in
Definition \ref{definition-bialgebra} is an example of
Situation \ref{situation-category}.  
The functor $\Gamma : \text{Mod}_H \to \text{Mod}_R$
is identified with the functor of invariants
$$
M \longmapsto M^H = \{m \in M \mid h \cdot m = \epsilon(h) m \}.
$$
Axioms (D), (I) and (S) hold.  
Axiom (C) holds if H is cocommutative.
\item The category $\text{Comod}_H$ with its additional structure introduced in
Definition \ref{definition-bialgebra} is an example of
Situation \ref{situation-category}.  
The functor $\Gamma : \text{Comod}_H \to \text{Mod}_R$ is
identified with the functor of coinvariants
$$
M \longmapsto M_H = \{m \in M \mid \sigma(m) = 1 \otimes m\}
$$
where $\sigma : M \to H \otimes_A M$ indicates the coaction of $M$. 
Axiom (D) holds.
Axiom (C) holds if $H$ is commutative.
\end{enumerate}
\end{lemma}

\begin{proof}
The first two statements in both part (1) and (2) are clear.
It also clear that axiom (D) holds in both cases.
Arbitrary direct products exist in the category $\text{Mod}_H$, 
which is axiom (I), and so by Lemma \ref{lemma-symmetric} axiom (S) holds.  
The final statement concerning axiom (C) is straightforward,
see \cite[Section 1.8]{montgomery}.
\end{proof}

\section{Adequacy for a bialgebra}

\noindent
Let $R \to A$ be map of commutative rings.
Let $H$ be a bialgebra over $A$.
Let $M$ be an $H$-module.
We can identify $\text{Sym}_H^n M := \text{Sym}^n_{\text{Mod}_H} M$ of axiom (S)  
with the $H$-module
$$
 \underbrace{M \otimes_A \cdots \otimes_A M}_{n} / M'
$$ 
where $M'$ is the submodule generated by elements 
$h \cdot ( (\cdots \otimes m_i \otimes \cdots \otimes m_j \otimes \cdots) -  
(\cdots \otimes m_j \otimes \cdots \otimes m_i \otimes \cdots))$ 
for $h \in H$ and $m_1, \ldots, m_n \in M$.
And $\text{Sym}_H M :=\bigoplus_n \text{Sym}^n_{H} M$ is the 
largest $H$-module quotient 
of the tensor algebra on $M$ which is commutative.

\medskip \noindent
An \emph{$H$-algebra} is an $H$-module $C$ which is an algebra 
over the algebra structure on $H$ such that $A \to C$ and 
$C \otimes_A C \to C$ are $H$-module homomorphisms.  We say 
that $C$ is \emph{commutative} if $C$ is commutative as an algebra.  
An $H$-module $M$ is \emph{locally finite} if it is the 
filtered colimit of finite type $H$-modules.

\medskip \noindent
The following definition reinterprets adequacy axiom (A) for the category $\text{Mod}_H$.  

\begin{definition} 
\label{definition-bialgebra-adequate}
Let $R \to A$ be map of commutative rings.
Let $H$ be a bialgebra over $A$.
We say that $H$ is \emph{adequate} if for every surjection of 
commutative $H$-algebras $C \to D$ in $\text{Mod}_H$ with $C$ locally 
finite, and any $f \in D^H$, there exists an $n > 0$ and an element $g \in C^H$ 
such that $g \mapsto f^n$ in $D^H$.
\end{definition}

\begin{lemma}
\label{lemma-adequate-bialgebra-equivalences}
Let $R \to A$ be map of commutative rings.
Let $H$ be a bialgebra over $A$.
The following are equivalent:
\begin{enumerate}
\item $H$ is adequate.
\item For every surjection of finite type $H$-modules
$N \to M$ and $f \in M^H$, there
exists an $n > 0$ and a $g \in (\text{Sym}_H^n N)^H$
such that $g \mapsto f^n$ in $(\text{Sym}_H^n M)^H$.
\end{enumerate}
If $A$ is Noetherian, then the above are also equivalent to:
\begin{enumerate} \setcounter{enumi}{2}
\item 
For every surjection of finite type $H$-modules
$N \to A$ and $f \in A$, there
exists an $n > 0$ and a $g \in (\text{Sym}_H^n N)^H$
such that $g \mapsto f^n$ in $A $.
\end{enumerate}
\end{lemma}

\begin{proof}  
This is Lemma \ref{lemma-adequate-equivalences}.
\end{proof}

\begin{corollary} \label{corollary-module-finite-type}
Let $R \to A$ be a finite type map of commutative rings where $R$ is Noetherian.  
Let $H$ be an adequate bialgebra over $A$.  
Let $C$ be a finitely generated, locally finite, commutative $H$-algebra. 
Then $C^H$ is a finitely generated $R$-algebra and for 
every finite type $C$-module M, the 
$C^H$-module $M^H$ is finite.
\end{corollary}

\begin{proof}
This is Theorem \ref{theorem-finite-type}.
\end{proof}

\begin{remark}  If $R = A = k$ where $k$ is a field, then \cite{kalniuk-tyc} define a
 Hopf algebra $H$ over $k$  to be \emph{geometrically reductive} if any finite 
 dimensional $H$-module 
$M$ and any non-zero homomorphism of $H$-modules $N \rightarrow k$ 
there exist $n > 0$ such that $\text{Sym}^n_H(N)^H \rightarrow k$ is non-zero.  
By Lemma \ref{lemma-adequate-bialgebra-equivalences}, 
$H$ is geometrically reductive if and only if $H$ is adequate. 

\medskip \noindent
In \cite[Theorem 3.1]{kalniuk-tyc}, Kalniuk and Tyc prove that with the 
hypotheses of the above corollary and with $R = A = k$ is a field, 
$C^H$ is finitely generated over $k$.  
\end{remark}

\section{Coadequacy for a bialgebra}

\noindent
Let $R \to A$ be map of commutative rings.
Let $H$ be a bialgebra over $A$.
An \emph{$H$-coalgebra} is an $H$-comodule $C$ which is an algebra 
over the algebra structure on $H$ such that $A \to C$ and 
$C \otimes_A C \to C$ are $H$-comodule homomorphisms;  
$C$ is \emph{commutative} if $C$ is commutative as an algebra.  
An $H$-comodule $M$ is \emph{locally finite} if it is the 
filtered colimit of finite type $H$-comodules.

\medskip \noindent
Here we reinterpret the adequacy axiom (A) 
for the category $\text{Comod}_H$.
\begin{definition} 
\label{definition-bialgebra-coadequate}
Let $R \to A$ be map of commutative rings.
Let $H$ be a bialgebra over $A$.
We say that $H$ is \emph{coadequate} if for every surjection 
of commutative $H$-coalgebras $C \to D$ with $C$ locally finite, 
and any $f \in D_H$, there exists an $n > 0$ 
and an element $g \in C_H$ such that $g \mapsto f^n$ in $D_H$.
\end{definition}

\noindent
Recall that we only know that axiom (S) holds for 
$\text{Comod}_H$ when $H$ is commutative.

\begin{lemma}
\label{lemma-coadequate-bialgebra-equivalences}
Let $R \to A$ be map of commutative rings.
Let $H$ be a commutative bialgebra over $A$.
The following are equivalent:
\begin{enumerate}
\item $H$ is adequate.
\item For every surjection of finite type $H$-modules
$N \to M$ and $f \in M^H$, there
exists an $n > 0$ and a $g \in (\text{Sym}_H^n N)^H$
such that $g \mapsto f^n$ in $(\text{Sym}_H^n M)^H$.
\end{enumerate}
If $A$ is Noetherian, then the above are also equivalent to:
\begin{enumerate} \setcounter{enumi}{2}
\item 
For every surjection of finite type $H$-modules
$N \to A$ and $f \in A$, there
exists an $n > 0$ and a $g \in (\text{Sym}_H^n N)^H$
such that $g \mapsto f^n$ in $A $.
\end{enumerate}
\end{lemma}

\begin{proof}  
This is Lemma \ref{lemma-adequate-equivalences}.
\end{proof}

\begin{corollary}
 \label{corollary-comodule-finite-type}
Let $R \to A$ be a finite type of commutative rings where $R$ is Noetherian.  
Let $H$ be an adequate bialgebra over $A$.  
Let $C$ be a finitely generated, locally finite,
 commutative $H$-coalgebra. 
Then $C_H$ is a finitely generated $R$-algebra 
and for every finite type $C$-module $M$, the 
$C_H$-module $M_H$ is finite.
\end{corollary}

\begin{proof}
This is Theorem \ref{theorem-finite-type}.
\end{proof}

\bibliography{references}{}
\bibliographystyle{alpha}

\end{document}